\documentclass[a4paper,11pt]{amsart}

\usepackage{amscd}
\usepackage{amssymb}
\usepackage{color}
\usepackage{tikz}
\newlength{\defbaselineskip} 
\usepackage{footnote}
\usepackage{color}
\usepackage{todonotes}

\newtheorem{thm}{Theorem}[section]
\newtheorem{cor}[thm]{Corollary}

\newtheorem{lemm}[thm]{Lemma}
\newtheorem{lem}[thm]{Lemma}
\newtheorem{prop}[thm]{Proposition}

\theoremstyle{definition}

\newtheorem{rem}[thm]{Remark}

\usepackage{fullpage}
\usepackage{tikz,tikz-cd}
\usetikzlibrary{matrix,arrows}
\makeatletter
\tikzset{
  edge node/.code={%
      \expandafter\def\expandafter\tikz@tonodes\expandafter{\tikz@tonodes #1}}}
\makeatother
\tikzset{
  subseteq/.style={
    draw=none,
    edge node={node [sloped, allow upside down, auto=false]{$\subseteq$}}},
  Subseteq/.style={
    draw=none,
    every to/.append style={
      edge node={node [sloped, allow upside down, auto=false]{$\subseteq$}}}
  }
}

 \numberwithin{equation}{section}
\numberwithin{equation}{section} \theoremstyle{definition}

          \newcommand\PP{{\mathbb{P}}}

           \newcommand\F{{\mathcal F}}

          \newcommand\oo{\mathcal O}
          \newcommand\Z{\mathbb{Z}}

\definecolor{zielony}{rgb}{0.5, 0.9, 0.1}
\definecolor{czerwony}{rgb}{0.8, 0.2, 0.1}
\definecolor{niebieski}{rgb}{0.3, 0.1, 0.9}

\newcounter{appendice}

\begin{document}
\title{EPW cubes}
\dedicatory{Dedicated to Piotr Pragacz on the occasion of his 60th birthday.}

\author[A.~Iliev]{Atanas Iliev}
\address{Seoul National University, Department of Mathematics, Gwanak Campus, Bldg. 27, Seoul 151-747, Korea}
\email{ailiev@snu.ac.kr}
\author[G.~Kapustka]{Grzegorz Kapustka}
\address{Institute of Mathematics of the Polish Academy of Sciences, ul. \'Sniadeckich 8, P.O. Box 21, 00-956 Warszawa, Poland}
\address{Jagiellonian University in Krak\'ow, ul. \L ojasiewicza 6, 30-348 Krak\'ow, Poland }
\email{grzegorz.kapustka@uj.edu.pl}
\author[M.~Kapustka]{Micha\l{} Kapustka}
\address{University of  Stavanger, Department of Mathematics and Natural Sciences, NO-4036 Stavanger, Norway}
\address{Jagiellonian University in Krak\'ow, ul. \L ojasiewicza 6, 30-348 Krak\'ow, Poland }
\email{michal.kapustka@uis.no}
\author[K.~Ranestad]{Kristian Ranestad}
\address{University of Oslo, Department of Mathematics, PO Box 1053, Blindern, N-0316 Oslo, Norway}
\email{ranestad@math.uio.no}

\begin{abstract}
We construct a new 20-dimensional family of projective 6-dimensional irreducible holomorphic symplectic manifolds. 
The elements of this family are deformation equivalent with the Hilbert scheme of three points on a K3 surface and are constructed as natural double covers of special codimension $3$ subvarieties of the Grassmanian $G(3,6)$. 
These codimension 3 subvarieties are defined
as Lagrangian degeneracy loci and their construction is parallel to that of EPW sextics, we call them the EPW cubes.
As a consequence we prove that the moduli space of polarized IHS sixfolds of K3-type, Beauville-Bogomolov degree 4 and divisibility 2 is unirational. 
\end{abstract}

\maketitle

\section{Introduction}
By an irreducible holomorphic symplectic (IHS) $2n$-fold we mean
a $2n$-dimensional simply connected compact K\"{a}hler manifold with trivial canonical
bundle that admits a unique (up to a constant)
closed non-degenerate holomorphic $2$-form and is not a product of two
manifolds (see \cite{Beauville}). The IHS manifolds are also known as hyperk\"ahler and irreducible symplectic manifolds, in dimension $2$ they are called $K3$ surface.

Moduli spaces of polarized K3 surfaces are a historically old subject, studied by the classical Italian geometers.
Mukai extended the classical constructions and proved unirationality results for the moduli spaces $\mathcal{M}_{2d}$ parametrizing polarized K3 surfaces of degree $2d$ for many cases with $d \leq19$ see \cite{M1}, \cite{M2}, \cite{M3}. On the other hand it was proven in \cite{GHS} that $\mathcal{M}_{2d}$ is of general type for
$d>61$ and some smaller values. Note that when the Kodaria dimension of such moduli space is positive the generic element of such moduli space is believed to be non-constructible. 

There are only five known descriptions of the moduli space of higher dimensional IHS manifolds (all these examples are deformations equivalent to $K3^{[n]}$). In dimension four we have the following unirational moduli spaces:
\begin{itemize}
\item double EPW sextics with Beauville-Bogomolov degree $q=2$ (see \cite{Ogrady-IHS}),
\item Fano scheme of lines on four dimensional cubic hypersurfaces with $q=6$ (see \cite{BeauvilleDonagi}),
\item $VSP(F,10)$ where $F$ define a cubic hypersurface of dimension $4$ with $q=38$ (see \cite{IlievRanestad}),
\item zero locus of a section of a vector bundle on $G(6,10)$ with $q=22$ described in \cite{DV}.
\end{itemize}
Moreover, there is only one more known family in dimension $8$ with $q=2$ studied in \cite{LehnLehnSorgervanStraten}.
Analogously to the case of $K3$ surfaces there are results in \cite{GritsenkoHulekSankaran} about the Kodaira dimension of the  moduli spaces of polarized IHS fourfolds of $K3^{[2]}$-type: In particular it is proven that such moduli spaces with split polarization of Beauville-Bogomolov degree $q\geq 24$ are of general type
(and for $q=18,22$ are of positive Kodaira dimension). We expect that the number of constructible families in higher dimension becomes small.

According to O'Grady \cite{Ogrady-IHS}, the $20$-dimensional family of natural double covers of special sextic hypersurfaces in $\PP^5$ (called EPW sextics) gives a maximal dimensional family
of polarized IHS fourfold deformation equivalent to the Hilbert scheme of two points on a $K3$-surface (this is a maximal dimensional family since $b_2(S^{[2]})=23$ for $S$ a $K3$-surface). 
%Only a few other maximal dimensional families of polarized IHS manifolds are known, see \cite{BeauvilleDonagi, DV, IlievRanestad, LehnLehnSorgervanStraten}. 
 Our aim is to perform a construction parallel to that of O'Grady to obtain a unirational $20$-dimensional family (also of maximal dimension) of polarized IHS sixfolds deformation equivalent to the Hilbert scheme of three points on a $K3$-surface (i.e.~of $K3^{[3]}$ type).
The elements of this family are natural double covers of special codimension $3$ subvarieties of the Grassmannian $G(3,6)$ that we call EPW cubes. 

Let us be more precise. Let $W$ be a complex $6$-dimensional vector space.
 We fix an isomorphism  $j: \wedge^6 W\to \mathbb{C}$ and the skew symmetric form
\begin{equation}\label{def of eta}\textstyle 
\eta: \wedge^3W\times \wedge^3 W\to \mathbb{C}, \quad (u,v)\mapsto j(u\wedge v).
 \end{equation} %Let $A$ be a ten dimensional Lagrangian space $A\subset \wedge^3W$ with respect to this form. 
 We denote by $LG_\eta(10, \wedge^3W)$  the variety of $10$-dimensional Lagrangian subspaces of $\wedge^3W$ with respect to $\eta$.
 For any $3$-dimensional subspace $U\subset W$, the $10$-dimensional subspace 
 \[\textstyle 
 T_U:=\wedge^2U\wedge W\subset \wedge^3W
 \]
 belongs to $LG_\eta(10, \wedge^3W)$, and $\PP(T_U)$ is the projective tangent space to $$G(3,W)\subset \PP(\wedge^3 W)$$ at $[U]$.
 
 For any  $[A]\in LG_\eta(10, \wedge^3W)$ and $k\in \mathbb{N}$, %be any $10$-dimensional Lagrangian space $A\subset \wedge^3W$. 
 we consider the following Lagrangian degeneracy locus, with natural scheme structure (see \cite{PragaczRatajski}),
 \[
 D_k^A=\{[U]\in G(3,W)\;|\; \dim A\cap T_U\geq k\}\subset G(3,W).
 \]
% with 
% $D_k^A$ be the scheme (see \cite{PragaczRatajski}) supported on the set of points $[U]\in G(3,W)$ such that $\PP(T_U)$ intersects $\PP(A)$ along a $(k-1)$-dimensional linear space.
For the fixed $[A]\in LG_\eta(10, \wedge^3W)$ we call  the scheme $D^A_2$ an \emph{EPW cube}. We prove that if $A$ is generic then $D_2^A$ is a sixfold singular only along the threefold  $D_3^A$ and that $D_4^A$ is empty. Moreover, $D_3^A$ is smooth such that the
singularities of $D_2^A$ are transversal $\frac{1}{2}(1,1,1)$ singularities along $D_3^A$.

Before we state our main theorem we shall need some more notation.
 The projectivized representation $\wedge^3$ 
 of $PGL(W)$ on $\wedge^3 W$
splits ${\PP}^{19} = {\PP}(\wedge^3 W)$ into a disjoint union of 4 orbits
$${\PP}^{19} = ({\PP}^{19}\setminus W) \cup (F\setminus  \Omega) \cup (\Omega \setminus G(3,W)) \cup G(3,W),$$
where $G(3,W) \subset \Omega \subset F \subset {\PP}^{19}$,
$\dim(\Omega) = 14$, $\operatorname{Sing}(\Omega) = G(3,W)$,
$\dim(F) = 18$, $\operatorname{Sing}(F) = \Omega$, see \cite{Donagi}.
We call the invariant sets $G,\Omega,F$ and ${\PP}^{19}$ the (projective)
orbits of $\wedge^3$ for $PGL(6)$. See \cite[Appendix]{GKapustkab23} for some results about the geometry 
of $\Omega$ and its relations with EPW sextics.
For any nonzero vector $w\in W$, denote by
$$F_{[w]}=\langle w\rangle \wedge (\wedge^2 W)$$ the 10-dimensional subspace of $\wedge^3 W$, such that 
$$\bigcup_{[w]\in \PP(W)}\PP(F_{[w]})= \Omega \subset \PP(\wedge^3 W).$$
We denote, after O'Grady \cite{EPWMichigan}, 
$$\Sigma =\{ [A]\in LG_\eta(10,\wedge^3W)|\quad  \PP(A)\cap G(3,W)\not=\emptyset \}$$ 
and
 $$\Delta=\{ [A]\in LG_\eta(10,\wedge^3W)| \quad \exists w\in W \colon \dim A\cap F_{[w]}\geq 3 \}. $$
We also consider a third subset
$$\Gamma= \{A\in LG_\eta(10,\wedge^3W)|\quad  \exists [U]\in G(3,W) \colon \dim A\cap T_U\geq 4\}.$$ 
 Denote by $$LG_\eta^1(10,\wedge^3W):=LG_\eta(10,\wedge^3W)\setminus (\Sigma\cup \Gamma).$$
All three subsets $\Sigma$, $\Delta$, $\Gamma$ are divisors (see \cite{EPWMichigan} and Lemma \ref{Gamma  div}) and $LG_\eta^1(10,\wedge^3W)$ is hence a dense open subset of $LG_\eta(10,\wedge^3W)$.
 Our main result is the following:
\begin{thm}\label{main}
 If $[A]\in LG^1_\eta(10, \wedge^3W)$, then there exists a natural double cover $Y_A$ of the EPW cube $D_2^A$ branched along its singular locus $D_3^A$ such that
 $Y_A$ is an IHS sixfold of $K3^{[3]}$-type with polarization of Beauville-Bogomolov degree $q=4$ and divisibility $2$. In particular, the moduli space of polarized IHS sixfolds of $K3^{[3]}$-type, Beauville-Bogomolov degree 4 and divisibility 2 is unirational. 
\end{thm}
We prove the theorem in Section \ref{proof} at the very end of the paper. 
 The plan of the proof is the following: 
 In Proposition \ref{existence and smoothness of double cover} we prove that for $[A]\in LG_\eta^1(10,\wedge^3W),$ the variety $D_2^A$ is singular only along the locus $D_3^A$ and that it admits a smooth double cover $Y_{A} \to D_2^A$ branched along $D_3^A$ 
%that is smooth and has 
with a trivial canonical class. The proof of the Proposition is based on a general study of Lagrangian degeneracy loci contained in Section \ref{general lagrangian loci}.  By globalizing the construction of the double cover to the whole affine variety $LG_\eta^1(10,\wedge^3W)$ we obtain a smooth family 
$$\mathcal{Y}\to LG^1_{\eta}(10,\wedge^3 W)$$
 with fibers $\mathcal{Y}_{[A]}=Y_A$. Note that the family $\mathcal{Y}$ is naturally a family of polarized varieties with the polarization given by the divisors defining the double cover.

In Lemma \ref{distinct delta gamma} we prove that $\Delta  \setminus (\Gamma\cup \Sigma)$ is nonempty. Following \cite[Section 4.1]{EPWMichigan}, we associate to a general $[A_0]\in\Delta  \setminus (\Gamma\cup \Sigma)$ a K3 surface $S_{A_0}$.   Then, in Proposition \ref{prop specialA}, we prove that
%for $[A]\in\Delta  \setminus (\Gamma\cup \Sigma)$  
there exists a rational $2:1$ map from the Hilbert scheme $S_{A_0}^{[3]}$ of length 3 subschemes on $S_{A_0}$ to the EPW cube $D_2^{A_0}$. We infer in Section \ref{proof} that in this case the sixfold $Y_{A_0}$ is birational to $S^{[3]}_{A_0}$.  Together with the fact that $Y_{A_0}$ is smooth, irreducible and has trivial canonical class, this proves that $Y_{A_0}$ is IHS. 
  
Since flat deformations of IHS manifolds are still IHS,  the family $\mathcal{Y}$ is a family of smooth IHS sixfolds.
The fact that the obtained IHS manifolds are of  $K3^{[3]}$-type is a straightforward consequence of Huybrechts theorem \cite[Thm.~4.6]{Huybrechts}. 

During the proof of Theorem \ref{main} we retrieve also some information on the constructed varieties. We prove in Section \ref{invariants} that the polarization $\xi$ giving the double cover $Y_A\to D_2^A$ has Beauville-Bogomolov degree $q(\xi)=4$ and is primitive. Moreover, the degree of an EPW cube $D^A_2\subset G(3,6)\subset\PP^{19}$  is $480$. 

Note that  the coarse moduli space $\mathcal{M}$ of polarized  IHS sixfolds of $K3^{[3]}$-type and Beauville-Bogomolov degree $4$ has two components distinguished by divisibility. 
We conclude the paper by proving that the image of the moduli map $LG^1_{\eta}(10,\wedge^3 W) \to \mathcal{M}$ defined by $\mathcal{Y}$ is a 20 dimensional open and dense subset of the component of $\mathcal{M}$ corresponding to divisibility 2 (see Proposition \ref{divisibility}).

{\bf Acknowledgements.} We thank Olivier Debarre, Alexander Kuznetsov and Kieran O'Grady for useful comments, O'Grady in particular for pointing out a proof of Proposition \ref{divisibility}. A. Iliev was supported by SNU grant 0450-20130016, G. Kapustka by NCN grant  2013/08/A/ST1/00312, M. Kapustka by NCN grant 2013/10/E/ST1/00688 and K. Ranestad by RCN grant 239015.

%It is a natural problem to study the degenerations of the  IHS fourfold (double EPW sextics) and 
%sixfold (double EPW cubes)  corresponding to $A\in \Gamma$; this will be addressed in a future paper.
  
%\input{Lagrangian}

 \begin{section}{Lagrangian degeneracy loci} \label{general lagrangian loci}
 In this section we study resolutions of Lagrangian degeneracy loci.
 Let us start with fixing some notation and definitions. We fix a vector space $W_{2n}$ of dimension $2n$ and a symplectic form $\omega\in \wedge^2 W_{2n}^*$.
 Let $X$ be a smooth manifold and let $\mathcal{W}=W_{2n}\times \mathcal{O}_X$ be the trivial bundle with fiber
 $W_{2n}$ on $X$ equipped with a nondegenerate symplectic form $\tilde{\omega}$ induced on each fiber by $\omega$.
 Consider $J\subset \mathcal{W}$ a Lagrangian vector subbundle i.e. a subbundle of rank $n$ whose fibers are isotropic
 with respect to $\tilde{\omega}$. Let $A\subset W_{2n}$ be a Lagrangian vector subspace inducing a trivial subbundle
 $\mathcal{A}\subset\mathcal{W}$.
 For each $k\in \mathbb{N}$ we define the set
 $$D^A_k=\{x\in X | \dim (J_x\cap \mathcal{A}_x)\geq k\}\subset X$$
 where $J_x$ and $\mathcal{A}_x$ denote the fibers of the bundles $J$ and $\mathcal{A}$ as subspaces in the fiber $\mathcal{W}_x$.
 Let us now define $LG_{\omega}(n,W_{2n})$ to be the Lagrangian Grassmannian parametrizing all subspaces of $W_{2n}$ which
 are Lagrangian with respect to $\omega$. Then $J$ defines a map $\iota:X \to LG_{\omega}(n,W_{2n})$ in such a way that
 $J=\iota^*\mathcal{L}$ where $\mathcal{L}$ denotes the tautological bundle on the Lagrangian Grassmannian $LG_{\omega}(n,W_{2n})$.
 Moreover, similarly as on $X$, we can define
 $$\mathbb{D}^A_k=\{[L]\in LG_{\omega}(n,W_{2n}) | \dim (L\cap A_{[L]})\geq k\}\subset LG_{\omega}(n,W_{2n}),$$
and $\mathbb{D}^A_k$ admits a natural scheme structure as a degeneracy locus.
 We then have $D^A_k=\iota^{-1} \mathbb{D}^A_k$, i.e. the scheme structure on $D^A_k$ is defined by the inverse image of the
 ideal sheaf of $\mathbb{D}^A_k$  \cite[p.163]{Hartshorne}.

 \subsection{Resolution of $\mathbb{D}^A_k$ }
 For each $k\in \mathbb{N}$, let $G(k,A)$ be the Grassmannian of $k$-dimensional subspaces of $A$ and let
 \[
 \tilde{\mathbb{D}}^A_k=\{([L],[U])\in LG_{\omega}(n,W_{2n})\times G(k,A) | L\supset U\}.
 \]
 By \cite{PragaczRatajski}, $\tilde{\mathbb{D}}^A_k$ is a resolution of $\mathbb{D}^A_k$.
 We shall describe the above variety more precisely.
 First of all we have the following incidence described more generally in \cite{PragaczRatajski}:

 \begin{center}
  \begin{tikzpicture}
  \matrix (m) [matrix of math nodes,row sep=3em,column sep=2em,minimum width=1em]
  { & \tilde{\mathbb{D}}^A_k&   \\
     \mathbb{D}_k^A & & G(k,A) \\};
  \path[-stealth]
    (m-1-2) edge node [left] {$\phi$} (m-2-1)
            edge node [right] {$\pi$} (m-2-3);
\end{tikzpicture}
\end{center}
The projection $\phi$ is clearly birational, whereas $\pi$ is a fibration with fibers isomorphic to a Lagrangian Grassmannian $LG(n-k,2n-2k)$.
In particular $\tilde{\mathbb{D}}^A_k$ is a smooth manifold of Picard number two with Picard group generated by
$H$, the pullback of the hyperplane section of $LG(n,W_{2n})$ in its Pl\"ucker embedding, and $R$, the pullback of the hyperplane section
of $G(k, A)$ in its Pl\"ucker embedding. Denote by $\mathcal{Q}$ the tautological bundle on $G(k,A)$ seen as a subbundle of the trivial symplectic
bundle $W_{2n}\otimes \oo_{G(k,A)}$.
Consider the subbundle $\mathcal{Q}^{\perp}\subset W_{2n}\otimes \oo_{G(k,A)}$ perpendicular to $\mathcal{Q}$ with respect the symplectic form.
The following was observed in \cite{PragaczRatajski}.
\begin{lem}
The variety $\tilde{\mathbb{D}}^A_k$ is isomorphic to the Lagrangian bundle 
\[
\F:=LG(n-k,\mathcal{Q}^{\perp}/\mathcal{Q} ).
\]
\end{lem}
Of course the tautological Lagrangian subbundle on $LG(n-k,\mathcal{Q}^{\perp}/\mathcal{Q})$ can be identified with the bundle
$\phi^{\ast}\mathcal{L}/\pi^{\ast}\mathcal{Q}=:\mathcal{W}$. In particular, we have $c_1(\mathcal{W})=
\phi^{\ast} c_1(\mathcal{L})-\pi^{\ast}c_1(\mathcal{Q})=R-H$.
\begin{lem}
 The relative tangent bundle $T_{\pi}$ of $\pi\colon \F \to G(k,A)$ is the bundle $S^2(\mathcal{W}^{\vee})$.
\end{lem}
\begin{proof}
 This can be seen by globalizing the construction of the tangent space of the Lagrangian Grassmannian described for example in
 \cite{Mukai}.
\end{proof}
\begin{lem}\label{canonical class for LG bundle}
 The canonical class of $\tilde{\mathbb{D}}^A_k$ is $-(n+1-k)H- (k-1) R$.
\end{lem}
\begin{proof}
 We use the exact sequence
 $$0\to T_{\pi} \to T_{\F} \to \pi^*T_{G(k,A)}\to 0.$$
Now $\mathcal{W}^{\vee}$ has rank $n-k$, so $$c_1(T_{\pi})=c_1(S^2(\mathcal{W}^{\vee}))=(n+1-k)c_1(\mathcal{W}^{\vee})=(n+1-k)(H-R)$$ while $\pi^* c_1(T_{G(k,A)})=nR$.
Hence $K_{\F}=-c_1(T_{\F})=-(n+1-k)H-(k-1)R$.
 \end{proof}

\begin{lem}\label{D_1 is hyperplane}
 The variety $\mathbb{D}^A_1$ is a hyperplane section of $LG_\omega(n,W_{2n})$.
\end{lem}
\begin{proof}
 Indeed $\mathbb{D}^A_1$ is the intersection of the codimension one Schubert cycle on the Grassmannian $G(n,2n)$ with the Lagrangian
 Grassmannian, hence a hyperplane section of the Lagrangian Grassmannian.
\end{proof}

 Let us denote by $\mathbb{E}$ the exceptional divisor of $\phi$.
 \begin{lem} \label{lem 2.5} For $k=2$ we have: $[\mathbb{E}]=[H]-2[R]$.
 \end{lem}
\begin{proof}
 It is clear that $[\mathbb{E}]= a [H] +b[R] $ for some $a,b\in \mathbb{Z}$. Let us now consider the restriction of $\mathbb{E}$ to a fiber of $\pi$ i.e. we fix
 $V_2\subset A$ a vector space of dimension 2 and consider $LG(n-2, V_2^{\perp}/V_2)$. Since $\mathbb{E}=\phi^{-1} D_{3}^A$ we have
 \[
 \mathbb{E}\cap \pi^{-1}[V_2]=\{[L]\in LG(n-2, V_2^{\perp}/V_2)| \dim (L/V_2\cap A/V_2)\geq 1\}.
 \]
  It is hence a divisor of type
 $\mathbb{D}^{A/V_2}_1$ which is a hyperplane section of the fiber by Lemma \ref{D_1 is hyperplane}. It follows that $a=1$.

 To compute the coefficient at $[R]$ we fix a subspace $V_{n-2}$ of dimension $n-2$ in $A$ and consider the Schubert cycle
 $$\sigma_{V_{n-2}}=\{[U]\in G(2,A)| \dim(U\cap V_{n-2})\geq 1\}.$$
 The class $[\sigma_{V_{n-2}}]$ in the Chow group of $G(2,A)$ is then the class of a hyperplane section.
 We now describe $\phi_*\pi^{*}(\sigma_{V_{n-2}})$ as the class of the Schubert cycle $\sigma_{n-2,n}$ on $LG(n,2n)$ defined by
 $$\sigma_{n-2,n}=\{[L]\in LG(n,2n)|\quad\dim(L\cap V_{n-2})\geq 1,\quad \dim(L\cap A)\geq 2\}.$$
 By \cite[Theorem 2.1]{PragaczRatajski} we have
 $$[\sigma_{n-2,n}]=c_1 (\mathcal{L}^{\vee})c_{3}(\mathcal{L}^{\vee})-2c_{4}(\mathcal{L}^{\vee}).$$
 Moreover, from the same formula \cite[Theorem 2.1]{PragaczRatajski} we have:
 $$[\mathbb{D}^A_2]= c_1(\mathcal{L}^{\vee})c_{2}(\mathcal{L}^{\vee})-2c_{3}(\mathcal{L}^{\vee}).$$
 In terms of intersection on $\tilde{\mathbb{D}}^A_2$ this gives
 $$H^{\frac{n(n+1)}{2}-3}\cap [\tilde{\mathbb{D}}^A_2]=c_1(\mathcal{L}^{\vee})^{\frac{n(n+1)}{2}-2}c_{2}(\mathcal{L}^{\vee})-2c_1(\mathcal{L}^{\vee})^{\frac{n(n+1)}{2}-3}
 c_{3}(\mathcal{L}^{\vee})$$
 and
 $$H^{\frac{n(n+1)}{2}-4}\cdot R\cap[\tilde{\mathbb{D}}^A_2]=c_1(\mathcal{L}^{\vee})^{\frac{n(n+1)}{2}-3}c_{3}(\mathcal{L}^{\vee})-2c_1(\mathcal{L}^{\vee})^{\frac{n(n+1)}{2}-4}
 c_{4}(\mathcal{L}^{\vee}).$$
 Since we know that $\mathbb{E}$ is contracted by the resolution to $\mathbb{D}^A_3$ we also have $\mathbb{E}\cdot H^{\frac{n(n+1)}{2}-4}=0$.
 We can now compute $b$:
 \begin{align}\label{computation of EH}
 0=&\mathbb{E}\cdot H^{\frac{n(n+1)}{2}-4}=(H+bR)\cdot H^{\frac{n(n+1)}{2}-4}=H^{\frac{n(n+1)}{2}-3}+bH^{\frac{n(n+1)}{2}-4}\cdot R=\\
 &c_1(\mathcal{L}^{\vee})^{\frac{n(n+1)}{2}-4}(c_1(\mathcal{L}^{\vee})^2c_{2}(\mathcal{L}^{\vee})+(b-2)c_1(\mathcal{L}^{\vee})
 c_{3}(\mathcal{L}^{\vee})-2b c_{4}(\mathcal{L}^{\vee})).
 \end{align}
Now, using the theorem of Hiller-Boe (\cite[Theorem 6.4]{PragaczLN1478}) on relations in the Chow ring of the Lagrangian Grassmannian we get
$$c_1(\mathcal{L}^{\vee})^2=2c_2(\mathcal{L}^{\vee}) \text{ and } c_2(\mathcal{L}^{\vee})^2=2(c_3(\mathcal{L}^{\vee})c_1(\mathcal{L}^{\vee})-c_4(\mathcal{L}^{\vee})).$$
Substituting in \ref{computation of EH} we get:
$$0=(b+2)\deg (c_1(\mathcal{L}^{\vee}) c_{3}(\mathcal{L}^{\vee})-2 c_{4}(\mathcal{L}^{\vee}))=(b+2)\deg \sigma_{n-2,n}.$$
It follows that $b=-2$.

\end{proof}

 \subsection{The embedding of $G(3,W)$ into $LG_\eta(10,\wedge^3 W)$} Let $W$ be a $6$-dimensional vector space.
   Let $G=G(3,W)\subset \mathbb{P}(\wedge^3 W)$ be the Grassmannian of $3$-dimensional subspaces in $W$
   in its Pl\"ucker embedding. Now, recall for each $[U]\in G$,  %$U\subset W$ be the $3$-space corresponding to $x$ and let 
   \[
   T_U=\wedge^2U\wedge W\subset \wedge^3 W.
   \]
    %denote the $10$-dimensional
   %subspace whose projectivization 
   $\PP(T_U)$ is tangent to $G(3,W)$ at $[U]$.  Let $\mathcal{T}$ be the corresponding vector subbundle of
   $\wedge^3 W\otimes \mathcal{O}_{G}$. Let $A$ be a $10$-dimensional subspace of
   $\wedge^3 W$ isotropic with respect to the symplectic form $\eta$ defined by (\ref{def of eta}) and such that $\PP(A)\cap G(3,W)=\emptyset$.
   Recall that for $k=1,2,3,4$ we defined 
   \[
   D^A_k=\{[U]\in G | \dim (T_U\cap A)\geq k\}\subset G.
   \]

  Observe that $\mathcal{T}$ is a Lagrangian subbundle of $\wedge^3 W \otimes \mathcal{O}_G$  with respect to
  the 2-form $\eta$. It follows that we are in the general situation
  described at the beginning of  Section \ref{general lagrangian loci}, with $n=10$, $W_{20}=\wedge^3 W$, $X=G$, $J=\mathcal{T}$ and $A=A$.
  Then  $\mathcal{T}$ defines a map  \[
  \iota: G(3,W) \to LG_{\eta}(10,\wedge^3 W),\quad [U]\mapsto [T_U].
  \]
  We denote by $\mathcal{C}_U:= \PP({T}_U)\cap G(3,W)$ the intersection of $G(3,W)$ with its projective tangent space $[U]$.  Then $\mathcal{C}_U$ is linearly isomorphic to a cone over $\mathbb{P}^2\times \mathbb{P}^2$ with vertex $[U]$.
  The quadrics containing the cone $\mathcal{C}_U$ plays in this situation a similar role in the local analyze of the singularities of $D_k^A$ as the Pl\"{u}cker quadrics containing the Grassmanian $\PP(F_{[w]})\cap G(3,W)$ in \cite{EPWMichigan}; this will be made more precise in Lemma \ref{quadricsC_v}.

   We aim at proving the following:
  \begin{prop} \label{transversality}  Let  $A\in LG_{\eta}(10,\wedge^3 W)$ such that $\PP(A)\cap G(3,W)=\emptyset$.
 
 The map $\iota$ is an embedding and $\iota(G(3,W))$ meets transversely all loci $\mathbb{D}^A_k\setminus \mathbb{D}^A_{k+1}$ for $k=1,2,3$. In particular each $D^A_{k}$ is of expected dimension.
\end{prop}

 For the proof we shall adapt the idea of \cite{EPWMichigan} to our context, that we first need to introduce.
Let us describe $\iota$ more precisely locally around a chosen point $[U_0]\in G(3,W)$. For this, we choose a basis $v_1,\dots,v_6$ for $W$ such that $U_0=\langle v_1, v_2, v_3\rangle$  and define
 $U_{\infty}=\langle v_4, v_5, v_6\rangle$. For any $[U]\in G(3,W)$ we have $T_U=\wedge^2 U\wedge W$, %where $U_x$ is the three-space corresponding to $x$, 
 so
 $T_{U_0}$, $T_{U_\infty}$ are two Lagrangian spaces that intersect only at $0$;  $T_{U_0}\cap T_{U_\infty}=0$. By appropriate choice of $v_4,v_5,v_6$ we can also assume that $T_{U_\infty}\cap A=0$.
 
 Let 
 \[
 \mathcal{V}=\{[L]\in LG_\eta(10,\wedge^3 W)| L \cap T_{U_\infty}=0\}.
 \]
 The decomposition $\wedge^3 W=T_{U_0}\oplus T_{U_\infty}$ into Lagrangian subspaces, and the isomorphism $T_{U_\infty}\to T_{U_0}^{\vee}$ induced by $\eta$, allows us to view a Lagrangian space $L$ in  $ \mathcal{V}$ as the graph of a symmetric linear map 
 $Q_L: T_{U_0}\to T_{U_\infty}=T_{U_0}^{\vee}$.  Let $q_L\in Sym^2 T_{U_0}^{\vee}$  be the quadratic form corresponding to $Q_L$.
 %, where $L_y$ denotes the Lagrangian subspace corresponding to $y$. 
 The map $[L]\mapsto q_L$ defines an isomorphism $\mathcal{V}\to Sym^2 T_{U_0}^{\vee}$.

 Consider the open neighbourhood 
 \[
 \mathfrak{U}=\{[U]\in G(3,W)| T_U\cap T_{U_\infty}=0\}
 \] 
 of $[U_0]$ in $G(3,W)$.  %then $\mathfrak{U}$ is an open subset of $G(3,W)$ and ${[U_0]}\in \mathfrak{U}$.
 For $[U]\in \mathfrak{U}$ we denote by $Q_U:=Q_{T_U}$ and $q_U:=q_{T_U}$ the symmetric linear map and the quadratic form corresponding to the Lagrangian space $T_U$.
 
 We shall describe $q_U$ in local coordinates.  Observe that  for any $[U]\in G(3,W)$,
  \[
T_U\cap T_{U_\infty}=0\leftrightarrow  U\cap {U_\infty}=0
 \] 
 and that any such subspace $U$ is the graph of a linear map $\beta_U:U_0\to U_{\infty}$.
 In particular, there is an isomorphism:
 \[
\rho: \mathfrak{U}\to Hom(U_0,U_{\infty}); \quad [U]\mapsto \beta_U
\]
whose inverse is the map
\[
 \alpha \mapsto [U_\alpha]:=[(v_1+\alpha(v_1))\wedge(v_2+\alpha(v_2))\wedge (v_3+\alpha(v_3))].
\]
In the given basis $(v_1,v_2,v_3), (v_4,v_5,v_6)$ for $U_0$ and $U_\infty$ we let $B_U=(b_{i,j})_{i,j\in\{1\dots 3\}}$ be the matrix of the linear map $\beta_U$. 
In the dual basis we let $(m_0, M)$, with $ M=(m_{i,j})_{i,j\in\{1\dots3\}}$, be the coordinates in
\[
T_{U_0}^{\vee}=(\wedge^3 U_0\oplus \wedge^2U_0\otimes U_\infty)^\vee=(\wedge^3 U_0\oplus Hom(U_0, U_\infty))^\vee
\]
%We let $(m_0, M)$, with $ M=(m_{i,j})_{i,j\in\{1\dots3\}}$ be the coordinates in $T_{U_0}^{\vee}$ associated to the choice
 % $(v_1,v_2,v_3), (v_4,v_5,v_6)$ are represented by pairs $(m_0, M)$, where $ M=(m_{i,j})_{i,j\in\{1\dots3\}}$.
Note, that under our identification the map $\iota:G(3,W)\to LG(10,wedge^3 W)$ restricted to $\mathfrak{U}$ is the map 
$ [U]\mapsto q_U$, which justifies our slight abuse of notation in the following.
  \begin{lem}\label{quadricsC_v}% Let $[U]\in \mathfrak{U}$ and let $\beta_U\in Hom(U_0,U_{\infty})$ be given by a $3\times 3$ matrix $B_U=(b_{i,j})_{i,j\in\{1\dots 3\}}$ in the basis $(v_1,v_2,v_3), (v_4,v_5,v_6)$. Points in $T_{U_0}^{\vee}$ in coordinates associated to the choice
%  $(v_1,v_2,v_3), (v_4,v_5,v_6)$ are represented by pairs $(m_0, M)$, where $ M=(m_{i,j})_{i,j\in\{1\dots3\}}$. 
  In the above coordinates, the map 
  \[
  \iota: \mathfrak{U}\ni [U]\mapsto q_U:=q_{T_U}\in Sym^2 T_{U_0}^{\vee}
  \]
  is defined by
  \begin{equation}\label{eqquadrics}
  q_U(m_0, M)=\sum_{i,j\in\{1\dots3\}} b_{i,j} M^{i,j} + m_0 \sum_{i,j\in\{1\dots3\}} B^{i,j}_U m_{i,j} + m_0^2 \det B_U, 
  \end{equation}
  where $M^{i,j}$, $B^{i,j}_U$ are the entries of the matrices adjoint to $M$ and $B_U$.
  \end{lem}
  \begin{proof} We write in coordinates the map $\wedge^3 U_0 \oplus \wedge^2 U_0\otimes U_{\infty} \to \wedge^3 U_{\infty} \oplus \wedge^2 U_{\infty}\otimes U_0  $
  whose graph is $ \wedge^3 U \oplus \wedge^2 U\otimes U_{\infty} $ where $U$ is the graph of the map $U_0\to U_{\infty}$ given by the matrix $B_U$.
  \end{proof}

  Let now $Q_A$ be the symmetric map  $T_{U_0}\to T_{U_\infty}=T_{U_0}^{\vee}$ whose graph is $A$ and $q_A$ the corresponding quadratic form.
  In this way
  $$D_l^A\cap\mathfrak{U}=\{[U]\in \mathfrak{U} |\dim T_U\cap A)\geq l\}=\{[U]\in \mathfrak{U} |\operatorname{rk} (Q_U-Q_A)\leq 10-l\},$$
  hence $D_l^A$ is locally defined by the vanishing of the $(11-l)\times (11-l)$ minors of the $10\times 10$ matrix with entries being polynomials in $b_{i,j}$.
  
  First we show that the space of quadrics that define %the tangent cone 
  ${\mathcal{C}_{U}}$,  surjects onto the space of quadrics on linear subspaces in $\PP(T_{U})$.
%\krtodo{I think I deleted some k here. It appears in the proof below, but I do not understand how, please check}
   \begin{lem}\label{restriction of quadrics to K} If $P\subset \PP(T_{U})\setminus G(3,6)$ is a linear subspace of dimension at most   $2$,  then the restriction map $\mathbf{r}_P: H^0(\PP(T_{U}), \mathcal{I}_{\mathcal{C}_{U}}(2)) \to H^0(P, \mathcal{O}_{P}(2))$ is surjective.
  \end{lem}
    \begin{proof} We may restrict to the case when  $P$ is a plane. Since % the tangent cone 
    ${\mathcal{C}_{U}}\subset\PP(T_{U}))\cap G(3,6)$ is projectively equivalent to the cone over $\PP^2\times \PP^2$ in its Segre embedding, it suffices to show that if $P\subset \PP^8$ is a plane that do not intersect 
    $\PP^2\times \PP^2\subset \PP^8$, then the Cremona transformation $\operatorname{Cr}$ on $\PP^8$ defined by the quadrics containing $\PP^2\times \PP^2$ maps $P$  to a linearly normal Veronese surface.  
    Note that the ideal of $\PP^2\times \PP^2\subset \PP^8$ is defined $2\times 2$ minors of a $3\times 3$ matrix with linear forms in $\PP^8$ and its secant by the determinant of this matrix. Since the first syzygies between the generators of this ideal are generated by linear ones we infer from \cite[Proposition~3.1]{AR} that they define a birational map.
     Moreover this Cremona transformation contracts the secant determinantal cubic hypersurface $V_3$, to a $\PP^2\times \PP^2$, so the the inverse Cremona is of the same kind.  Furthermore, the fibers of the map $V_3\to \PP^2\times \PP^2$ are $3$-dimensional linear spaces spanned by quadric surfaces in $\PP^2\times \PP^2$.  Now, by assumption, $P$ does not intersect $\PP^2\times \PP^2$, so the restriction $\operatorname{Cr}|_P$ is a regular, hence finite, morphism. Since the fibers of the Cremona transformation are linear, $P$ intersects each fiber in at most a single point, so the restriction $\operatorname{Cr}|_P$ is an isomorphism.   Thus, if  $\operatorname{Cr}(P)$ is not linearly normal, the linear span $\langle \operatorname{Cr}(P)\rangle$ is a $\PP^4$, being a smooth projected Veronese surface.  Assume this is the case.  Then $\operatorname{Cr}(P)$ is not contained in any quadric. Since the quadrics that define the inverse Cremona, map $\operatorname{Cr}(P)$ to the plane $P$, these quadrics form only a net, when restricted to the $4$-dimensional space $\langle \operatorname{Cr}(P)\rangle$.  In fact the complement of $\PP^2\times \PP^2\cap \langle \operatorname{Cr}(P)\rangle$ in $\langle \operatorname{Cr}(P)\rangle$ is mapped to $P$ by the inverse Cremona transformation.  Therefore $\langle \operatorname{Cr}(P)\rangle$ must be contained in the cubic hypersurface that is contracted by this inverse Cremona.  Since this hypersurface is contracted to the original $\PP^2\times \PP^2$, we infer that $P$ is contained in $\PP^2\times \PP^2$.  This contradicts our assumption and concludes our proof.
    % (   This is proven by a Macaulay 2 computation for any $P$ such that $\PP(P)\cap \mathcal{C}_{\mu}=\emptyset$.)
    \end{proof}
  \begin{lem} \label{tangent cone} Let $K=A\cap T_{U_0}=\ker Q_A\subset T_{U_0}$ and assume that $k=\dim K\leq 3$. %Assume that $k_0\leq 3$. 
  Then for any $l\leq k$ the tangent cone $\mathfrak{C}^l_{A,U_0}$ of $D_l^A\cap \mathfrak{U}$ at $U_0$ is linearly isomorphic to a cone over the corank $l$ locus  of quadrics in $\PP(H^0(\PP (K),\mathcal{O}_{\PP (K)}(2)))$.
  \end{lem}
  \begin{proof} We follow the idea of \cite[Proposition 1.9]{OGradyTaxonomy}. If we choose a basis $\Lambda$ of $T_{U_0}^{\vee}$, the symmetric linear map $Q_U$ is defined by a symmetric matrix $M^{\Lambda}(B_U)$ with entries being polynomials in $(b_{i,j})_{i,j\in \{1\dots 3\}}$.
  
%  matrix Observe that by choosing a basis $\Lambda$ of $T_{U_0}^{\vee}$ we can think of the map $\iota: \mathfrak{U}\ni [U]\to Q_U\in Sym^2 T_{U_0}^{\vee}$ as a symmetric matrix $M^{\Lambda}(B_U)$ with entries being polynomials in coordinates $(b_{i,j})_{i,j\in \{1\dots 3\}}$. 
  
  The linear summands of each entry in $M^{\Lambda}(B_U)$ form a matrix that we denote by  $N^{\Lambda}(B_U)$.
  % with entries given by linear parts of entries of $M^{\Lambda}(B_U)$. 
  Since $Q_0=0$, the entries of $M^{\Lambda}(B_U)$ have no nonzero constant terms. 
 Moreover, by using Lemma \ref{quadricsC_v} and $\Lambda_0=(m_0,M)$, we see that the map $\mathfrak{U}\ni U\mapsto q'_U\in Sym^2 T_{U_0}^{\vee}$, where  $q'_U$ is the quadratic form corresponding to the symmetric map defined by the matrix $N^{\Lambda_0}(B_U)$, maps $\mathfrak{U}$ linearly onto the linear system of quadrics containing the cone $\mathcal{C}_{U_0}$. Of course, this surjection is independent of the choice of basis.

 We now choose a basis $\Lambda$ in $T_{U_0}$ in which $Q_A$ is represented by a diagonal matrix $R_{k}={\rm diag}\{0\dots 0, 1\dots 1\}$ with $k$ zeros in the diagonal.  
% Since the 
% matrix $N^{\Lambda}$ arises from $N^{\Lambda_0}$ by change of basis it follows that the map $\iota^N: \mathfrak{U}\ni [U]\to Q'_U\in Sym^2 T_{U_0}^{\vee}$ associated to $N^{\Lambda}(B_U)$ still maps $\mathfrak{U}$ linearly onto the linear system of quadrics containing the cone $\mathcal{C}_{U_0}$. 
 Then  \begin{equation*}
 \begin{split}
 D_l^A\cap\mathfrak{U}=\{[U]\in \mathfrak{U}| \dim (T_U\cap A)\geq l \}=\{[U]\in \mathfrak{U} | \dim \operatorname{ker} (Q_U-Q_A) \geq l \}\\=\{[U]\in \mathfrak{U}| \operatorname{rank} (M^{\Lambda}(B_U)-R_{k}) \leq 10-l \}.
 \end{split}
 \end{equation*}
 Hence $D_l^A$ is defined in coordinates $(b_{i,j})_{i,j\in \{1\dots 3\}}$ on $\mathfrak{U}$ by $(11-l)\times (11-l)$ minors of the matrix  $M^{\Lambda}(B_U)-R_{k}$. 
Furthermore, since $[U_0]$ is the point $0$ in our coordinates $(b_{i,j})_{i,j\in \{1\dots 3\}}$, the tangent cone to $D_l^A\cap\mathfrak{U}$ at  $[U_0]$ is defined by the initial terms of the $(11-l)\times (11-l)$ minors of $M^{\Lambda}(B_U)-R_{k}$. 
Note that we can write 
$$M^{\Lambda}(B_U)-R_{k}=-R_{k}+N^{\Lambda}(B_U)+ Z(B_U),$$
where the entries of the matrix $Z(B_U)$ are polynomials with no linear or constant terms. We illustrate this decomposition as follows.
%{\tiny\[
%\left(
%\begin{array}{c|c}
%  N_{\mathbf{k}}^{\Lambda}+ Z_{\mathbf{k}} 
%& \begin{array}{cccc} N_{1,k+1}^{\Lambda}+ Z_{1,k+1}^{\Lambda} & \dots & N_{1,10}^{\Lambda}+ Z_{1,10}^{\Lambda} \\
%\vdots &\ddots&\vdots\\
%N_{k,k+1}^{\Lambda}+ Z_{k,k+1}^{\Lambda}& \dots & N_{k,10}^{\Lambda}+ Z_{k,10}^{\Lambda}
%\end{array} \\ \hline
%  \begin{array}{ccc} N_{k+1,1}^{\Lambda}+ Z_{k+1,1}^{\Lambda} & \dots & N_{k+1,k}^{\Lambda}+ Z_{k+1,k}^{\Lambda} \\
%\vdots &\ddots&\vdots\\
%N_{10,1}^{\Lambda}+ Z_{10,1}^{\Lambda} & \dots & N_{10,k}^{\Lambda}+ Z_{10,k}^{\Lambda}

%  \end{array} & \begin{array}{cccc}  
%1+   N_{k+1,k+1}^{\Lambda}+ Z_{k+1,k+1}^{\Lambda}& N_{k+1,k+2}^{\Lambda}+ Z_{k+1,k+2}^{\Lambda}&\dots & N_{k+1,10}^{\Lambda}+ Z_{k+1,10}^{\Lambda} \\
%N_{k+2,k+1}^{\Lambda}+ Z_{k+2,k+1}^{\Lambda}&1+   N_{k+2,k+2}^{\Lambda}+ Z_{k+2,k+2}^{\Lambda}&\dots & N_{k+2,10}^{\Lambda}+ Z_{k+2 ,10}^{\Lambda}\\
%\vdots & &\ddots & \vdots\\
%N_{10,k+1}^{\Lambda}+ Z_{10,k+1}^{\Lambda} & \dots & \dots & 1+ N_{10,10}^{\Lambda}+ Z_{10,10}^{\Lambda} 

%  \end{array}\\
%\end{array}
%\right)
%\]}

{\tiny
\[
\left(
\begin{array}{c|c}
  N_{\mathbf{k}}^{\Lambda}+ Z_{\mathbf{k}} 
& \begin{array}{cccc} N_{1,k+1}^{\Lambda}+ Z_{1,k+1}^{\Lambda} & \dots & N_{1,10}^{\Lambda}+ Z_{1,10}^{\Lambda} \\
\vdots &\ddots&\vdots\\
N_{k,k+1}^{\Lambda}+ Z_{k,k+1}^{\Lambda}& \dots & N_{k,10}^{\Lambda}+ Z_{k,10}^{\Lambda}
\end{array} \\ \hline
  \begin{array}{ccc} N_{k+1,1}^{\Lambda}+ Z_{k+1,1}^{\Lambda} & \dots & N_{k+1,k}^{\Lambda}+ Z_{k+1,k}^{\Lambda} \\
\vdots &\ddots&\vdots\\
N_{10,1}^{\Lambda}+ Z_{10,1}^{\Lambda} & \dots & N_{10,k}^{\Lambda}+ Z_{10,k}^{\Lambda}

  \end{array} & \begin{array}{cccc}  
-1+   N_{k+1,k+1}^{\Lambda}+ Z_{k+1,k+1}^{\Lambda}&\dots & N_{k+1,10}^{\Lambda}+ Z_{k+1,10}^{\Lambda} \\
%N_{k+2,k+1}^{\Lambda}+ Z_{k+2,k+1}^{\Lambda}&\dots & N_{k+2,10}^{\Lambda}+ Z_{k+2 ,10}^{\Lambda}\\
\vdots  &\ddots & \vdots\\
N_{10,k+1}^{\Lambda}+ Z_{10,k+1}^{\Lambda}  & \dots & -1+ N_{10,10}^{\Lambda}+ Z_{10,10}^{\Lambda} 

  \end{array}\\
\end{array}
\right)
\]
}
Let $\Phi$ be an  $(11-l)\times (11-l)$ minor of $M^{\Lambda}(B_U)-R_{k}$ and consider its decomposition $\Phi=\Phi_0+\dots +\Phi_r$ 
into homogeneous parts $\Phi_d$ of degree $d$.
Observe that $\Phi_d=0$ for $d\leq {k}-l$, moreover $\Phi_{{k}-l+1}$ can be nonzero only if the sub matrix associated to the minor $\Phi$ contains all nonzero entries of $R_k$.
In the latter case  $\Phi_{{k}-l+1}$ is a $({k}+1-l)\times ({k}+1-l)$ minor of the ${k}\times {k} $ upper left corner sub matrix $N_{\mathbf{k}}^{\Lambda}(B_U)$ of the matrix $N^{\Lambda}(B_U)$. Let us now denote by $q'_U$ the quadric corresponding to the matrix $N^{\Lambda}(B_U)$ and by $\iota^N$ the map $U\mapsto q'_U$.
Then, by changing $\Phi$ we get that  the tangent cone of $D_l^A\cap\mathfrak{U}$ is contained in:
$$\hat{\mathfrak{C}}^l_{A,U_0}:=\{[U]\in \mathfrak{U} | \operatorname{rank} (N_{\mathbf{k}}^{\Lambda}(B_U)) \leq {k}-l\}=
\{[U]\in \mathfrak{U} | \operatorname{rank}(q'_U|_{K})\leq {k}-l\}.$$
The latter is the preimage by $\mathbf{r}_{K}\circ \iota^N$ of the corank $l$ 
locus in the projective space of quadrics $\PP( H^0(\PP (K),\mathcal{O}_{\PP (K)}(2)))$. By Lemma \ref{restriction of quadrics to K}, we have seen that 
$\mathbf{r}_{K}\circ \iota^N$ is a linear surjection.  So we conclude that $\hat{\mathfrak{C}}^l_{A,U_0}$ is a cone over the corank $l$ locus of quadrics in $\PP(H^0( \PP (K), \mathcal{O}_{\PP (K)}(2)))$ with vertex a linear space of dimension $10-\frac{k(k+1)}{2}$. 
It follows that $\hat{\mathfrak{C}}^l_{A,U_0}$ is an irreducible variety of codimension $\frac{l(l+1)}{2}$ equal to the codimension  of $D_l^A$. Thus we have equality $\mathfrak{C}^l_{A,U_0}=\hat{\mathfrak{C}}^l_{A,U_0}$ which ends the proof.
\end{proof}
  
\begin{cor} \label{cor of 2.9} If $A$ is a Lagrangian space in $\wedge^3 W$, such that $\PP(A)$ doesn't meet $G(3,W)$, then the variety $D_l^A$ is smooth of the expected codimension $\frac{l(l+1)}{2}$ outside $D^A_{l+1}$.  
Moreover, if $l=2$ and $\dim  A\cap T_{U_0}=3$, i.e. $[U_0]$ is a point in  $D^A_3\setminus D^A_4$, then the tangent cone $\mathfrak{C}^2_{A,U_0}$ is a cone over the Veronese surface in $\PP^5$ centered in the  tangent space of $D^A_3$.
\end{cor}
 \begin{proof} [Proof of Proposition \ref{transversality}] It is clear from Lemma \ref{quadricsC_v} that $\iota$ is a local isomorphism into its image, and by  Corollary \ref{cor of 2.9}, the subscheme $D_A^k =\iota^{-1}(\iota(G(3,W))\cap\mathbb{D}_A^k)$ is smooth outside $D^{k+1}_A$, so $\iota (G(3,W)$ meets the degeneracy loci transversally.
\end{proof}

\subsection{Invariants} \label{invariants}
We shall compute the classes of the Lagrangian degeneracy loci 
$D^A_k\subset G(3,W)$ in the Chow ring of $G(3,W)$.  We consider the embedding  $\iota: G(3,W)\to LG_\eta(10,\wedge^3 W)$ defined by the bundle of Lagrangian subspaces $\mathcal{T}$ on $G(3,W)$.  According to  \cite[Theorem 2.1]{PragaczRatajski} the fundamental classes of the Lagrangian degeneracy loci 
$D^A_k$ are
\[
[D^A_1]=[c_1({\mathcal T}^{\vee})\cap G(3,W)] , \quad [D^A_2]=[(c_2c_1-2c_3)({\mathcal T}^{\vee})\cap G(3,W)]
\]
and 
\[
[D^A_3]=[(c_1c_2c_3-2c_1^2c_4+2c_2c_4+2c_1c_5-2c_3^2)({\mathcal T}^{\vee})\cap G(3,W)].
\]
The $\PP^9$-bundle $\PP(\mathcal{T})$ is the projective tangent bundle on $G(3,W)$.
So $\mathcal{T}^{\vee}$ fits into an exact sequence
\[
0\to\Omega_{G(3,W)}(1)\to \mathcal{T}^{\vee} \to \mathcal{O}_{G(3,W)}(1)\to 0
\]
 and we get
\[
\deg D^A_1=168,  \quad \deg D^A_2=480,\quad \deg D^A_3=720
\]
\begin{rem}
This may be compared with the degree of the line bundle $2H-3E$ on $S^{[3]}$, where $S$ is a K3 surface of degree 10, $H$ is the pullback of the line bundle of degree $10$ on $S$, and $E$ is the unique divisor class such that the divisor of non-reduced subschemes in $S^{[3]}$ is equivalent to $2E$.   The degree, i.e. the value of the Beauville Bogomolov form is $q(2H-3E)=4$, and the degree and the Euler-Poincare characteristic of the line bundle is 
\[
(2H-3E)^6=15 q(2H-3E)^3=960 \quad {\rm and} \quad\chi(2H-3E)=10.
\]
So if the map defined by $|2H-3E|$ is a morphism of degree $2$, the image would have degree $480$, like $D^A_2$.
\end{rem}
In the section \ref{special-special}, we show that $S^{[3]}$ for a general $K3$-surface $S$ of degree $10$, admits a rational double cover of a degeneracy locus $D^A_2$.  However that double cover is not a morphism.

\section{The double cover of an EPW cube}
 
 \begin{prop} \label{existence and smoothness of double cover} Let $[A]\in LG_\eta(10,\wedge^3W)$.  If $\mathbb{P}(A)\cap G(3,W)=\emptyset$ and $D^A_4=\emptyset$, then $D^A_2$ admits a double cover $f: Y_A\to D^A_2$ 
  branched over $D^A_3$  with $Y_A$ a smooth irreducible manifold having trivial canonical class. 
 \end{prop}
Before we pass to the construction of the double cover let us observe the following.
 \begin{lem}\label{irreducible D2} Under the assumptions of Proposition \ref{existence and smoothness of double cover} the variety $D^A_2$ is integral.
\end{lem}
\begin{proof}  We know that $D^A_2$ is of expected dimension. Observe now that by Corollary \ref{cor of 2.9} the variety $D^A_2$ is irreducible if and only if it is connected. To prove connectedness we perform a computation in the Chow ring of the Grassmannian $G(3,W)$ showing that the class $[D^A_2]$ does not decompose into a sum of nontrivial effective classes in the Chow group $A^3(G(3,W))$ whose intersection is the zero class in $A^6(G(3,W))$. More precisely we compute:
$$[D^A_2]=16 h^3-12 h s_2+12 s_3$$
where $h$ is the hyperplane class on $G(3,W)$, $s_2$ and $s_3$ are the Chern classes of the tautological bundle on $G(3,W)$.
We then solve in integer coordinates $a,b,c \in \mathbb{Z}$ the equation
$$(a h^3- b s_2+ c s_3)((16-a) h^3-(12-b) s_2+(12-c) s_3)=0$$
in the Chow group $A^6(G(3,W))$ which is generated by: $s_2^3$, $h^3s_1s_2 $, $s_3^2$.
Multiplying out the equation in the Chow ring and extracting coefficients at the generators we get a system of three quadratic diophantine equations in $a,b,c$:

\begin{equation}
\begin{cases}
- 5a^2  + 4ab - b^2  + 56a - 20b=0\\
- 6a^2  + 8ab - 2b^2  - 4ac + 2bc + 72a - 52b + 20c=0\\ 
6a^2  - 6ab + b^2  + 2ac - c^2  - 72a + 36b - 4c=0\\
 \end{cases}     
 \end{equation}
 The only integer solutions are: $(0,0,0)$ and $(16, 12, 12)$. This ends the proof.  
\end{proof}

 The plan of the construction of the double cover in Proposition \ref{existence and smoothness of double cover}  is the following. We consider the resolution $\tilde{D}_2^A\to D^A_2$  with exceptional divisor $E$.  We prove that $E$ is a smooth even divisor, and hence that  there is a smooth double cover $\tilde{Y}\to \tilde{D}_2^A$ branched over $E$.  Finally, we contract the branch divisor of the double cover using a suitable multiple of the pullback of a hyperplane class on $D^A_2$ by the resolution and the double cover. 
 
 Thus, we start by defining the incidences
\[
 \tilde{D}_2^A=\{([U],[U'])\in G(3,W)\times G(2,A)| \quad T_U \supset U'\},
 \]
 and 
 \[\textstyle 
 \tilde{\mathbb{D}}_2^A=\{([L],[U'])\in LG_\eta(10,\wedge^3W)\times G(2,A)| \quad L \supset U'\}.
 \]
They fit in the following diagram:
\begin{center}
\begin{tikzcd}
 G(3,W)              \rar{\iota}   &  LG_{\omega}(10,\wedge^3 W)\\
 D_2^A   \rar[swap]{\iota|_{D_2^A}} \arrow[Subseteq]{u}{}   & \mathbb{D}^A_2 \arrow[Subseteq]{u}{}\\
\tilde {D}_2^A \uar[swap]{\alpha} \rar[swap]{\tilde{\iota}} & \tilde{\mathbb{D}}^A_2 \uar[swap]{\phi}
\end{tikzcd}
\end{center}

\begin{lem} \label{smooth blow up and exceptional divisor} Under the assumptions of Proposition \ref{existence and smoothness of double cover} the variety $\tilde{D}_2^A$ as well as the exceptional locus $E$ of the map $\alpha$ are smooth. In particular $\alpha$ is a resolution of singularities of $D_2^A $.\end{lem}
\begin{proof} Since we know that $D_4^A=\emptyset$,  the resolution $\alpha: \tilde{D}_2^A\to D_2^A$ is just the blow up of $D_2^A $ along $D_3^A$. Now,  $\tilde{D}_2^A\setminus E$ is isomorphic to $D_2^A\setminus D^A_3$, so, by Corollary \ref{cor of 2.9}, we deduce that $\tilde{D}_2^A$ is smooth outside $E$. 
Let $p\in E\subset \tilde{D}_2^A$. Then  $\alpha(p)\in D_3^A$.  Take $\mathbf{P}_1,\mathbf{P}_2,\mathbf{P}_3$ to be three general hyperplanes passing through $\alpha(p)$. Consider $Z_{\mathbf{P}}={D}_2^A\cap \mathbf{P}_1\cap \mathbf{P}_2 \cap \mathbf{P}_3$ and its strict transform  $\tilde{Z}_{\mathbf{P}} \subset  \tilde{D}_2^A$. We have the following diagram:

\begin{center}
\begin{tikzcd}
 \tilde{Z}_{\mathbf{P}}    \dar{\alpha_{\mathbf{P}}}           \rar[swap]  &  \tilde{D}_2^A \dar{\alpha} \\
Z_{\mathbf{P}}    \rar[swap]   & D^A_2\end{tikzcd}
\end{center}
 The map $\alpha_{\mathbf{P}}: \tilde{Z}_{\mathbf{P}}  \to Z_{\mathbf{P}} $ is the blow up of $Z_{\mathbf{P}} $ in $D_3^A\cap \mathbf{P}_1\cap \mathbf{P}_2 \cap \mathbf{P}_3 $, which by Corollary \ref{cor of 2.9} is a finite set of isolated points.  By the assumption on $\mathbf{P}_1,\mathbf{P}_2,\mathbf{P}_3$ the strict transform $\tilde{Z}_{\mathbf{P}}$ contains the whole fiber $\alpha^{-1}(p)$ and hence also $p\in \tilde{Z}_{\mathbf{P}}$.  Let $\tilde {\mathbf{P}}_i$ be the strict transform of $\mathbf{P}_i$ for $i=1,2,3$.  Then $\tilde {\mathbf{P}}_i$ is a Cartier divisor on  $\tilde{D}_2^A$  and $ \tilde{Z}_{\mathbf{P}}=\tilde {\mathbf{P}}_1\cap \tilde {\mathbf{P}}_2\cap\tilde {\mathbf{P}}_3$ is a complete intersection of Cartier divisors on $\tilde{D}_2^A$.
 Now, from Corollary \ref{cor of 2.9},  the exceptional divisor $E_\mathbf{P}=E \cap \tilde{Z}_{\mathbf{P}}$ of $\alpha_{\mathbf{P}}$ is isomorphic to a finite union of disjoint $(\mathbb{P}^2)'s$, one for each point in $D_3^A\cap \mathbf{P}_1\cap \mathbf{P}_2 \cap \mathbf{P}_3 $. But  $E_\mathbf{P}$ is itself a Cartier divisor on $\tilde{Z}_{\mathbf{P}}$ by general properties of blow up.  Therefore $\tilde{Z}_{\mathbf{P}}$ is smooth. We conclude that $\tilde{D}_2^A$ is smooth at $p$ and similarly, that $E$ is smooth at $p$.
 % Proposition \ref{transversality} and Lemma \ref{tangent cone} we infer that $D_3^A$ is smooth and the normal cone of $D_3^A$ in $D_2^A$ is relatively a cone over $\PP^2$.  The exceptional locus $E$ is fibred with $\mathbb{P}%^2$`s over $D_3^A$. Now, $E$ is a smooth Cartier divisor (by general properties of the blowing-up) hence $\tilde{D}_2^A$ is smooth in points of $E$. Furthermore $\tilde{D}_2^A\setminus E \simeq D_2^A\setminus D_3^A$ is also %smooth. We deduce that $\tilde{D}_2^A$ is smooth.
%The last assertion is now also a consequence of Lemma \ref{tangent cone} since $D_2^A$ has transversally $\frac{1}{2}(1,1,1)$ singularities.
\end{proof}
We compute the first Chern class of the normal bundle of the embedding  $\tilde {\iota}: \tilde{D}_2^A\to \tilde {\mathbb{D}}_2^A$.
  \begin{lem}
  % We have the following: 
   $$c_1(\tilde{\iota}^*N_{\tilde{\iota}(\tilde{D}_2^A)|\tilde{\mathbb{D}}_2^A})=c_1(\alpha^*\iota^*N_{\iota(G(3,W))|LG_\eta(10,\wedge^3 W)})=38h,$$
   where $h$ is the pullback via the resolution $\alpha$ of the restriction of the hyperplane class on $G(3,W)$ to $D^A_2$.
  \end{lem}
  \begin{proof} From the transversality (Proposition \ref{transversality}) we have 
  $$\tilde{\iota}^*N_{\tilde{\iota}(\tilde{D}_2^A)|\tilde{\mathbb{D}}_2^A)}=\alpha^*\iota^*N_{\iota(G(3,W))|LG_\eta(10,\wedge^3 W)}.$$
  which gives the first equality.
  
  To get the second, consider the exact sequence:
  $$0\to T_{G(3,W)}\to \iota^*(T_{LG_\eta(10,\wedge^3 W)})\to \iota^*(N_{\iota(G(3,W))|LG_\eta(10,\wedge^3 W)})\to 0,$$
  and observe that 
  $\iota^*(T_{LG_\eta(10,\wedge^3 W)})=\iota^* (S^2 \mathcal{L}^{\vee})=S^2(\iota^* \mathcal{L}^{\vee})=S^2 \mathcal{T}^{\vee}$, where $\mathcal{L}$ denotes, as before, the tautological bundle on the Lagrangian Grassmannian $LG_\eta(10,\wedge^3 W)$.
  We obtain $$c_1(\alpha^*\iota^*N_{\iota(G(3,W))|LG_\eta(10,\wedge^3 W)})=-11 \alpha^* c_1(\mathcal{T})-6h.$$
 Now, from 
$$ 0\to\mathcal{O}_{G(3,W)}(-1)\to \mathcal{T}\to T_{G(3,W)}(-1)\to 0$$
we obtain $\alpha^* c_1(\mathcal{T})=-4h$, which proves the lemma. 
  \end{proof}
Note that in our notation we have $\tilde{\iota}^* H=\tilde{\iota}^*\phi^*c_1(\mathcal{L}^{\vee})=\alpha^* \iota^* c_1(\mathcal{L}^{\vee})=\alpha^* c_1(\mathcal{T}^{\vee})=4h$.
We aim now at constructing a double covering of $\tilde{D}_2^A$ branched along $E$. It is enough to prove that $E$ is an even divisor.
This follows from the exact sequence:
$$0\to T_{\tilde{D}_2^A}\to\tilde{\iota}^* T_{\tilde{\mathbb{D}}_2^A}  \to \tilde{\iota}^*N_{\tilde{\iota}(\tilde{D}_2^A)|\tilde{\mathbb{D}}_2^A}\to 0,$$
and Lemma \ref{canonical class for LG bundle}. Indeed, from them we infer
$$c_1(T_{\tilde{D}_2^A})=\tilde{\iota}^*(9H+R)-38h=\tilde{\iota}^*(R)-2h,$$
which, by Lemma \ref{lem 2.5}, means $E=\mathbb{E} \cap \tilde{D}_2^A=\tilde{\iota}^*(H-2R)=2K_{\tilde{D}_2^A}$. 
By  Lemma \ref{smooth blow up and exceptional divisor} there hence exists a smooth double cover  $\tilde{f}: \tilde{Y}\to \tilde{D}_2^A$ branched along the exceptional locus $E$ 
of the resolution $\alpha$. Moreover, from the adjunction formula for double covers we get $K_{\tilde{Y}}=\tilde{f}^{-1} (E)=: \tilde{E}$. 

%Since $\tilde{E}'$ is isomorphic to $E'$ we infer that it is also a $\PP^2$ fibration from Lemma \ref{smooth blow up and exceptional divisor}. The normal bundle is easy found from the local description of a double cover.
We now need to contract $\tilde{E}=\tilde{f}^{-1}(E)$ on $\tilde{Y}$.
For that, with slight abuse of notation, we denote by $h$  the class of the hyperplane section on $D_2^A\subset G(3,W)$. Then $|\tilde{f}^*\alpha^* h|$ is a globally generated linear system whose associated morphism defines $\alpha \circ \tilde{f}$ and hence contracts $E$ to a threefold and is 2:1 on $\tilde{Y}\setminus \tilde{f}^{-1}(E)$. It follows by standard arguments (for example applying Stein factorization and \cite[Proposition 4.4]{HartshorneLN})  that there exists a number $n$ such that  the system $|n \tilde{f}^*\alpha^* h|$ defines a morphism $\tilde{\alpha}: \tilde{Y} \to Y $ which is a birational morphism contracting exactly $\tilde{E}$ to a threefold $\mathcal{Z}$ and such that its image $Y$ is normal. We then have the following diagram

\begin{center}
\begin{tikzcd}
 \tilde{Y}    \dar{\tilde{\alpha}}           \rar{\tilde{f}}  &  \tilde{D}_2^A \dar{\alpha} \\
Y    \rar{f}   & D^A_2
\end{tikzcd}
\end{center}
in which $Y$ admits a 2:1 map $f:Y\to D^A_2$ branched along $D^A_3$. 
\begin{proof}[Proof of Proposition \ref{existence and smoothness of double cover}]We have constructed $Y$, a normal variety admitting a 2:1 map $f:Y\to D^A_2$ branched along $D^A_3$. Clearly $K_{\tilde{Y}}=\tilde{E}$ implies $K_Y=0$. It hence remains to prove that $Y$ is smooth. Since $\tilde{\alpha}$ is a contraction that contracts only $\tilde{E}$ it is clear that  $Y$ is smooth outside of $\mathcal{Z}=\tilde{\alpha}( \tilde{E})$. Let now $p\in \mathcal{Z}$ and  let $p'=f(p)$. We then choose three general hypersurfaces $\mathbf{P}_1,\mathbf{P}_2, \mathbf{P}_3$ of degree $n$ in $\PP(\wedge^3 W)$ passing through $p'$. Consider $Z_{\mathbf{P}}=D_2^A \cap \mathbf{P}_1\cap \mathbf{P}_2\cap  \mathbf{P}_3$ and $Z'_{\mathbf{P}}=D_3^A \cap \mathbf{P}_1\cap \mathbf{P}_2\cap  \mathbf{P}_3$.  Then  $Z'_{\mathbf{P}}$ is a finite set of points that includes $p'$.  Consider the following natural restriction of the above diagram:

\begin{center}
\begin{tikzcd}
 \tilde{Y}_{\mathbf{P}}   \dar{\tilde{\alpha}_{\mathbf{P}}}           \rar{\tilde{f}_{\mathbf{P}}}  &  \tilde{Z}_{\mathbf{P}} \dar{\alpha_{\mathbf{P}}} \\
Y_{\mathbf{P}}    \rar{{f}_{\mathbf{P}}}   & Z_{\mathbf{P}}
\end{tikzcd}
\end{center}
Here  $\alpha_{\mathbf{P}}=\alpha|_{\alpha^{-1} (Z_\mathbf{P})}: \tilde{Z}_{\mathbf{P}} \to {Z}_{\mathbf{P}} $ is just the blow up of $Z_\mathbf{P}$ along $Z'_{\mathbf{P}}$.  The exceptional divisor $E_{\mathbf{P}}$ is then, by Corollary \ref{cor of 2.9}, isomorphic to a finite set of disjoint $({\mathbb{P}^2})'s$ that each have normal bundle $\mathcal{O}_{\mathbb{P}^2}(-2)$ in $\tilde{Z}_{\mathbf{P}}$.
 Taking the double cover of $\tilde{Z}_{\mathbf{P}}$ branched 
 along the exceptional divisor $E_{\mathbf{P}}$, the preimage of these  $({\mathbb{P}^2})'s$ are the components of $\tilde{E}_{{\mathbf{P}}} \subset \tilde{Y}_{\mathbf{P}}$, each component a ${\mathbb{P}^2}$   with normal bundle  
 $\mathcal{O}_{\mathbb{P}^2}(-1)$. The contraction $\tilde{\alpha}_{\mathbf{P}}$ contracts the divisor
 $\tilde{E}_{{\mathbf{P}}}$ to a finite set of points in $Y_{\mathbf{P}}$.   It contracts one of its $({\mathbb{P}^2})'s$, denote it by $\tilde{E}_{{\mathbf{P}}}^{p}$, to the point $p$. 
 Note also that from the construction, $Y_{\mathbf{P}}$ is the intersection of three Cartier divisors on
 $Y$ which is smooth outside the finite set of points $Z'_{\mathbf{P}}$. Thus, since we constructed $Y$ 
 to be normal, we deduce that $Y_\mathbf{P}$ is also normal. We claim that $p$ must be a smooth point of $Y_{\mathbf{P}}$. Indeed, we know that $\tilde{\alpha}_{\mathbf{P}}$ 
 is a birational morphism onto the normal variety $Y_\mathbf{P}$. Moreover, all lines $\mathbf{l}\subset \tilde{E}_{{\mathbf{P}}}^p=\mathbb{P}^2$ are numerically equivalent on $\tilde{Y}_{\mathbf{P}}$
 and satisfy $\mathbf{l}\cdot K_{\tilde{Y}_{\mathbf{P}}}=-1<0$. It follows from \cite[Corollary 3.6]{Mori}, that there exists an extremal ray $r$ for $\tilde{Y}_{\mathbf{P}}$ 
 whose associated contraction $\operatorname{cont}_r:\tilde{Y}_{\mathbf{P}}\to \hat{Y}_{\mathbf{P}}$ contracts  $\tilde{E}_{{\mathbf{P}}}^{p}$ to a point $\hat{p}$ and that
 $\tilde{\alpha}_{\mathbf{P}}$ factorizes through $\operatorname{cont}_r$. By \cite[Theorem 3.3]{Mori} we have that $\operatorname{cont}_r$ is the blow down of $\tilde{E}_{{\mathbf{P}}}^p$
 and $\hat{p}$ is a smooth point of $\hat{Y}_{\mathbf{P}}$. Let us now denote by $\sigma: \hat{Y}_{\mathbf{P}}\to Y_{\mathbf{P}}$ the morphism satisfying $\tilde{\alpha}_{\mathbf{P}}=\sigma\circ
 \operatorname{cont}_r$. Consider $\sigma_o$ the restriction of $\sigma$ to small open neighborhoods of $\hat{p}$ and $p$.  Then $\sigma_o$ is a birational proper morphism which is bijective to an open subset of the 
  normal variety $Y_{\mathbf{P}}$. It follows by Zariski Main Theorem that $\sigma_o$ is an isomorphism and in consequence, $p$ is a smooth point 
 on $Y_{\mathbf{P}}$.
 
 The latter implies that $Y$ must also be smooth 
 at $p$ as it admits a smooth complete intersection
 subvariety which is smooth at $p$.
 \end{proof}
 
\begin{cor}\label{co} Let $[A]\in LG_\eta(10,\wedge^3W)$ be a general Lagrangian subspace with a $3$-dimensional intersection with some  space $F_{[w]}=\{w\wedge \alpha|\ \alpha \in \wedge ^2 W\}$, then 
there exists a double cover $f_A: Y_A\to D_2^A$ branched over $D_3^A$, where $Y_A$ is a smooth irreducible sixfold with trivial canonical class.
\end{cor}
\begin{proof} It is enough to make a dimension count to prove that the general Lagrangian space $A$ satisfying the assumptions of the Corollary also satisfies the assumptions of Proposition \ref{existence and smoothness of double cover}. 
Indeed, let as in the introduction
\[\textstyle
\Delta =\{[A]\in LG_\eta(10,\wedge^3W)| \exists w\in W: \dim (A\cap F_{[w]})\geq 3\},
\]
 and 
 \[\textstyle
 \Gamma =\{[A]\in LG_\eta(10,\wedge^3W)| \exists U \in G(3, W): \dim (A\cap T_U)\geq 4\}.
 \] 
%O'Grady showed in \cite{EPWMichigan}[Proposition 2.2] that $\Delta$ is a divisor distinct from the divisor $\Sigma$ defined above.
We show:
\begin{lemm} \label{Gamma  div} The set $\Gamma\subset LG_\eta(10,\wedge^3W)$ is a divisor.
%Moreover, the dimension of $\Delta$ is also $54$.
\end{lemm}
\begin{proof} Let us consider the incidence
$$\textstyle \Xi=\{ ([U],[A])\in G(3,W)\times LG_\eta(10,\wedge^3W): \dim(T_U\cap A)\geq 4 \}.$$
The dimension of $\Xi$ can be computed by looking at the projection $\Xi \to G(3,6)$.
For a fixed tangent plane we choose first a $\PP^3$ inside: this choice has $24$ parameters.
Then for a fixed $\PP^3$ we have $\dim(LG(6,12))=21$ parameters for the choice of $A$.
Thus the dimension of $\Xi$ is $9+24+21=54$. It remains to observe that the projection $\Xi \to LG_\eta(10,\wedge^3W)$ is finite, and that  $\dim (LG_\eta(10,\wedge^3W))=55$.
%The divisor $\Delta$ was studied in \cite{EPWMichigan}.
\end{proof}

Note that in \cite[Proposition 2.2]{EPWMichigan} it is proven that $\Delta$ is irreducible and not contained in $\Sigma=\{[A]\in LG(10,20) |  \PP(A)\cap G(3,W)\neq \emptyset\}$.
Our corollary is now a consequence of  Proposition \ref{existence and smoothness of double cover} and the following lemma. 
\begin{lem} \label{distinct delta gamma} The divisors $\Delta$, $\Gamma \subset LG_\eta(10,\wedge^3 W)$ have no common components. \end{lem}
\begin{proof}
We need to prove $\dim (\Delta\cap\Gamma)< 54$ which, by the fact that $\Delta$ is irreducible and not contained in $\Sigma$, is equivalent to $ \dim ((\Delta\cap \Gamma)\setminus \Sigma)< 54$. For this, observe that 
if $[A]\in(\Delta\cap\Gamma)\setminus \Sigma$ then there exist $[U]\in G(3,W)$ and $[w]\in \mathbb{P}(W)$ with $\dim (A\cap T_U)=4$ and $\dim (A\cap F_{[w]})=3$. 
We can hence consider the incidence:
%\begin{equation}
%\begin{split}
%\textstyle \Theta=\{([A],[W_3],[W_4],[w],[U])|\quad W_3 \subset A\cap F_x, W_4\subset A \cap T_U\}\\
%\hskip 3pt \subset LG_\eta(10, \wedge^3 W) \times G(3, \wedge^3 W)\times G(4, \wedge^3 W)\times \mathbb{P}(W)\times G(3,W)
%\end{split}
%\end{equation}
\begin{align}
 \Theta=&\{([A],[W_3],[W_4],[w],[U])\;|\; W_3 = A\cap F_{[w]}, W_4= A \cap T_U\}\\
 &\subset LG_\eta(10, \wedge^3 W) \times G(3, \wedge^3 W)\times G(4, \wedge^3 W)\times \mathbb{P}(W)\times G(3,W)\nonumber
\end{align}
such that its projection  to $LG_\eta(10, \wedge^3 W)$ contains  $(\Delta\cap\Gamma)\setminus \Sigma$. Note also that if we take $([A],[W_3],[W_4],[w],[U])\in \Theta$ then 
$W_4\cap W_3=W_4\cap F_{[w]}=W_3 \cap T_{U}$.

We shall now compute the dimension of $\Theta$ by considering fibers under subsequent projections:
\begin{multline*}\textstyle
 LG_\eta(10, \wedge^3 W) \times G(3, \wedge^3 W)\times G(4, \wedge^3 W)\times \mathbb{P}(W)\times G(3,W)\\ \textstyle  \xrightarrow{\pi_1}  G(3, \wedge^3 W)\times G(4, \wedge^3 W)\times \mathbb{P}(W)\times G(3,W) \xrightarrow{\pi_2} G(4, \wedge^3 W)\times \mathbb{P}(W)\times G(3,W)  \\ \xrightarrow{\pi_3} \mathbb{P}(W)\times G(3,W) \end{multline*}
We have two possibilities for pairs $([w],[U])$ which give us two types of points to consider:
\begin{enumerate}
\item $w \not \in U$, then $\dim T_U\cap F_{[w]}= 3$. 
\item $w \in U$, then $\dim T_U\cap F_{[w]}= 7$.
\end{enumerate}
We then have different types of elements in the intersection  $\pi_3^{-1}([w],[U])\cap \pi_2(\pi_1(\Theta))$, depending on 
 the number $d_1:=\dim (W_4 \cap F_{[w]})= \dim (W_4 \cap W_3)\leq 3$.  If $W_4^{\perp}$ denotes the orthogonal to $W_4$ w.r.t. $\eta$ in $\wedge^3W$, then $\dim W_4^{\perp}\cap F_{[w]}=6+d_1$. Now, in order for $[W_3]$ to be an element of $\pi_2^{-1}([W_4],[w],[U]) \cap \pi_1(\Theta)$
 we must have $W_3\subset W_4^{\perp}\cap F_{[w]}$.   
  %If we denote $d_2=\dim W_4 \cap W_3$ this gives different types of elements in $\pi_2^{-1}([W_4],[w],[U]) \cap \pi_2\circ \pi_1(\Theta)$. 
The fiber $\pi_1^{-1}([W_3],[W_4],[w],[U])\cap \Theta$ is of dimension $\frac{(3+d_1)(4+d_1)}{2}$. Hence to compute the dimension of each component of $\Theta$ it is enough to compute the dimensions of the spaces $F_{i,d_1}$ of elements $([W_3],[W_4],[w],[U])$ of types $(i,d_1)$, where $i=1$ if $w \not \in U$ and  $i=2$ if $w \in U$.
\begin{enumerate}
\item For $i=1$ we start with a choice of $[U]\in G(3,W)$. Then $[w]$ belongs to an open subset of $\mathbb{P}^5$. We have $d_1\leq 3$ and $[W_4]$ belongs to the Schubert cycle consisting of 4-spaces in the $10$-dimensional space $T_U$ that meet the fixed 3-space $T_U\cap F_{[w]}$ in dimension $d_1$. And $[W_3]$ belongs to the Schubert cycle of 3-spaces in the $(6+d_1)$-dimensional space $W_4^{\perp}\cap F_{[w]}$ that contains the space $W_4 \cap F_{[w]}$ of dimension $d_1$ %in dimension $d_2$.
\item For $i=2$ we again start with a choice of $[U]\in G(3,W)$. In this case $[w]$ belongs to $\PP(U)$. We have $d_1\leq 3$ and $[W_4]$ belongs to the Schubert cycle of 4-spaces in the $10$-dimensional space $T_U$ that meet the fixed 7-space $T_U\cap F_{[w]}$ in dimension $d_1$. Then $[W_3]$ belongs to the Schubert cycle of 3-spaces in the $(6+d_1)$-dimensional space $W_4^{\perp}\cap F_{[w]}$ that contains the space $W_4 \cap F_{[w]}$ of dimension $d_1$.% in dimension $d_2$.
\end{enumerate}
We have:
\begin{equation*}
\dim F_{i,d_1}=\begin{cases} 9+5+d_1(3-d_1)+(4-d_1)6+ (d_1+3)(3-d_1) \\ \hskip 5cm =47-3d_1-2d_1^2 \hskip 0.5cm &\text{ for i=1,} \\
9+2+d_1 (7-d_1)+ (4-d_1)6 + (d_1+3)(3-d_1) \\
\hskip 5cm  = 44+d_1-2d_1^2 &\text{ for i=2.}
\end{cases}
\end{equation*}
In each case we have $\dim F_{i,d_1}+\frac{(3+d_1)(4+d_1)}{2}\leq 53$. % except when $i=2, d_1=2, d_2<2$ ???, in which case we get $54$.
It follows that $\dim \Theta \leq 53$ which implies $\dim (\Delta\cap \Gamma)\leq 53$. 
Hence $\Delta$ and $\Gamma$ have no common components. 
\end{proof}
This concludes the proof also of Corollary \ref{co}. \end{proof}
 \end{section}

 \section{Special $A$}\label{special-special}

Let us recall from \cite{EPWMichigan} the following construction.
Let $V$ and $V_0$ be two vector spaces of dimensions 5 and 1 respectively. Let $W=V\oplus V_0$.
Consider the space $\wedge^3 W$ equipped with the symplectic form $\eta$ given by the wedge product as above. 
%As above 
%\begin{itemize}
%\item  for each $w\in W$ we denote $F_{[w]}:=\{\alpha \in \wedge^3 W | \alpha\wedge w=0\}=\langle w\rangle\wedge(\wedge^2 W) $
%\item for each $[U]\in G(3,W)$ we denote $T_{U}:= \wedge^2 U \wedge W\subset \wedge^3 W$ the affine cone of the projective tangent space to $G(3,W) $ in $[U]$
%\end{itemize}
Let $v_0\in V_0$, choose $A$ a general Lagrangian subspace of $\wedge^3 W$ such that $A\cap F_{[v_0]}$ is a vector space of dimension 3 i.e. $[A]$ is a general element of the divisor $\Delta\subset LG_\eta(10,\wedge^3W)$. In particular, we assume $[A]\in \Delta\setminus \Sigma$. Note that, by \cite[Proposition 2.2 (2)]{EPWMichigan}, for a general $[A]\in \Delta$ there is a unique $[v_0]$ such that $F_{[v_0]}\cap A$ is of dimension 3.

Let $\tilde{K}=A\cap F_{[v_0]}$ and denote by $K\subset \wedge^2 V$ the $3$-dimensional subspace such that $\tilde{K}=v_0\wedge K$.
Observe that there is a natural isomorphism $\wedge^2 V \to F_{[v_0]}$ given by wedge product with $v_0$. The latter induces  an isomorphism $\wedge^3 V \to F_{[v_0]}^{\vee}$.

Let $[B]\in LG_\eta(10,\wedge^3W)$ be a Lagrangian space such that $B\cap F_{[v_0]}=\{0\}$ and $B\cap A=\{0\}$.
Then the symplectic form $\eta$ defines a canonical isomorphism $B\to F_{[v_0]}^{\vee}$ by which $A$ appears as the graph of a symmetric map
$\tilde{Q}_A: F_{[v_0]} \to B=F_{[v_0]}^{\vee}$. Composed with the isomorphisms $\wedge^2 V \to F_{[v_0]}$ and $\wedge^3 V \to F_{[v_0]}^{\vee}$
we get a symmetric map 
$$Q_A: \textstyle \wedge^2 V \to \wedge^3 V\cong (\wedge^2 V)^\vee.$$

Clearly $\ker Q_A=K$. Let $q_A$ be the quadric on $\wedge^2 V$ given by $Q_A$, then $q_A$ is a quadric of rank 7; it is a cone over $K$.
The map $Q_A$ defines an isomorphism $\wedge^2 V/K \to K^{\perp}$ and hence the quadric $q_A$ defines a quadric $K^{\perp}\subset \wedge^3 V$:
\[
q_A^*:
%K^{\perp}\to \C;\quad 
\beta\mapsto \operatorname{vol}(\alpha\wedge \beta), \quad{\rm where}\quad Q_A(\alpha)=\beta.
\]  
Moreover, to each $v^*\in V^{\vee}$ we associate the quadric:
$$\textstyle q_{v^*}:\wedge^3 V \ni \omega \mapsto \operatorname{vol} (\omega(v^*)\wedge \omega) \in \mathbb{C}.$$
The quadrics  $q_{v^*}$ are the Pl\"ucker quadrics defining the Grassmannian $G(3, V)\subset \mathbb{P}(\wedge^3 V) $.
We denote by $S_A$ the smooth  K3 surface  (see \cite[Corollary 4.9]{EPWMichigan}) of genus 6 defined on $\mathbb{P}(K^{\perp})$ by the restrictions of the quadrics $q_{v^*}$ and the quadric $q_A^*$.
Let $S_A^{[2]}$ and $S_A^{[3]}$ denote the appropriate Hilbert schemes of points on $S_A$.
Observe that we have a natural isomorphism:
$$W^{\vee}=V^{\vee}\oplus V_0^{\vee} \in v^*+cv^*_0 \mapsto q_{v^*}+ c q_A^*\in H^0(\mathcal{I}_{S_A}(2))$$
We then have a rational two to one map:
\[\varphi : S_A^{[2]} \dashrightarrow \mathbb{P}(W)\]
well defined on the open subset consisting of reduced subschemes whose span is not contained in $G(3, V)$, by associating to $\{\beta_1,\beta_2\}\subset S_A$ the hyperplane in $W^{\vee}=H^0(\mathcal{I}_S(2))$ consisting of quadrics containing
the line $\langle\beta_1,\beta_2\rangle$.
Let us describe this map more precisely.  Since $\{\beta_1,\beta_2\}\subset K^{\perp}\subset \wedge^3 V\subset \wedge^3 W$,
 $\beta_i \wedge \kappa=0$ for $i=1,2$ and $\kappa\in K$ hence also for $\kappa\in \tilde{K}$. Thus $\beta_i\in \wedge^3 W$  is contained
in the space spanned by $A$ and $F_{[v_0]}$. It follows that there exist $\alpha_i \in \wedge^2 V$ such that $\beta_i+v_0\wedge \alpha_i \in A$. 
%This means, in particular, that $q^*_A(\lambda_1\beta_1+(\lambda_2\beta_2)=\alpha_i\wedge $
Let us fix such $\alpha_i$ (determined up to elements in $K$).  Then $Q_A(\alpha_i)=\beta_i$ and
\[
q^*_A(\lambda_1\beta_1+\lambda_2\beta_2)=\operatorname{vol}((\lambda_1\alpha_1+\lambda_2\alpha_2)\wedge(\lambda_1\beta_1+\lambda_2\beta_2))=\lambda_1\lambda_2\operatorname{vol}(\alpha_1\wedge\beta_2+\alpha_2\wedge \beta_1)
\]
since $q^*_A(\beta_1)=q^*_A(\beta_2)=0$.
But $A$ is Lagrangian, so we %(or equivalently, $q_A^*(\beta_i)=0, i=1,2 $)
have:
%\begin{itemize}
%\item $\alpha_i\wedge\beta_i=0$ for $i=1,2$
%\item $\operatorname{vol}(\alpha_1\wedge\beta_2)=\operatorname{vol}(\alpha_2\wedge \beta_1):=c_{12}$
%\end{itemize}
\[
\alpha_i\wedge\beta_i=0\quad  i=1,2\quad {\rm and }\quad 
\operatorname{vol}(\alpha_1\wedge\beta_2)=\operatorname{vol}(\alpha_2\wedge \beta_1):=c_{12}.
\]
Now, $\beta_1$ and $\beta_2$ are decomposable, i.e. $q_{v^*}(\beta_i)=0$, and their linear span is not contained in $G(3,V)$. We may therefore choose a basis $\{v_1,...,v_5\}$ for $V$ such that $\beta_1=v_1\wedge v_2\wedge v_3$ and $\beta_2=v_1\wedge v_4 \wedge v_5 $.  A direct computation now shows 
%\[
%q_{v_i^*}(\lambda_1\beta_1+\lambda_2\beta_2)=0;\quad i=2,...,5, 
%\]
%\[
%q_{v_1^*}(\lambda_1\beta_1+\lambda_2\beta_2)=2\lambda_1\lambda_2\quad {\rm and}\quad q^*_A(\lambda_1\beta_1+\lambda_2\beta_2)=2c_{12}\lambda_1\lambda_2
%\]
\[
(t_0 q^*_A+\sum_{i=1}^5 t_iq_{v_i^*})(\lambda_1\beta_1+\lambda_2\beta_2)=2t_0c_{12}\lambda_1\lambda_2+
2t_1\lambda_1\lambda_2\quad
\]
so 
\begin{equation}\label{eq explicit phi }
 \varphi( \{\beta_1,\beta_2\})=[c_{12}v_0+ v_1]\in \PP(W).
\end{equation}
It is proven in \cite{EPWMichigan} that $\varphi( \{\beta_1,\beta_2\})$ lies on the EPW sextic associated to $A$. Let us present the proof in a way that we will be able to further generalize.
It suffices to show that there are nonzero  scalars $x_1,x_2$ and an element $\kappa\in K$, such that 
$$(x_1(\beta_1+v_0\wedge \alpha_1)+x_2(\beta_2+v_0\wedge \alpha_2)+v_0\wedge\kappa)\wedge (c_{12}v_0+ v_1)=0.$$
Indeed, this implies $[x_1(\beta_1+v_0\wedge \alpha_1)+x_2(\beta_2+v_0\wedge \alpha_2)+v_0\wedge\kappa] \in \mathbb{P}(F_{[c_{12}v_0+ v_1]})\cap \mathbb{P}(A)$.
Let us now denote by $\kappa_1$, $\kappa_2$, $\kappa_3$ a basis of $K$, then we consider the equation 
\[
(x_1(\beta_1+v_0\wedge \alpha_1)+x_2(\beta_2+v_0\wedge \alpha_2)+\sum_{j=1}^3y_jv_0\wedge\kappa_j)\wedge (c_{12}v_0+ v_1)=0.
\]
i.e.
\[
(-x_1c_{12}v_0\wedge\beta_1-x_2c_{12}v_0\wedge\beta_2+x_1v_0\wedge \alpha_1\wedge v_1+x_2v_0\wedge \alpha_2\wedge v_1+\sum_{j=1}^3y_jv_0\wedge\kappa_j\wedge v_1)=0.
\]
To make this equation into a system of linear equations we multiply with the elements of basis in $\wedge^2V$ and compose with the volume map $\operatorname{vol}:\wedge^6W\to \mathbb{C}$.  

We obtain trivial equations when multiplying by $v_1\wedge v_i, i=2,3,4,5$.  
Multiplying with $v_2\wedge v_3$ we get 
\[
\kappa_i\wedge v_1\wedge v_2\wedge v_3=\kappa_i\wedge \beta_1=0, i=1,2,3,
\]
\[
\beta_1 \wedge v_2\wedge v_3=0,  \alpha_1\wedge v_1 \wedge v_2\wedge v_3=\alpha_1\wedge\beta_1=0
\]
 and 
 \[
 \alpha_2\wedge v_1 \wedge v_2\wedge v_3=\alpha_2\wedge\beta_1=c_{12}=c_{12}\operatorname{vol}(v_0\wedge\beta_2\wedge v_2\wedge v_3).
 \]  

So the equation multiplied with $v_2\wedge v_3$ is also trivial.  Similarly, the equation multiplied with $v_4\wedge v_5$ is trivial.  So the only nontrivial linear equations are obtained by multiplying by forms in $\langle v_2\wedge v_4, v_2\wedge v_5, v_3\wedge v_4, v_3\wedge v_5\rangle$.  Each of these 2-vectors annihilates $\beta_1$ and $\beta_2$, so we get the following four independent equations in $5$ variables, with a unique solution up to scalars:
\begin{align}
(x_1 \alpha_1+x_2 \alpha_2+\sum_{j=1}^3y_j\kappa_j)\wedge v_0\wedge v_1\wedge v_2\wedge v_4&=0.\nonumber\\
(x_1 \alpha_1+x_2 \alpha_2+\sum_{j=1}^3y_j\kappa_j)\wedge v_0\wedge v_1\wedge v_2\wedge v_5&=0.\nonumber\\
(x_1 \alpha_1+x_2 \alpha_2+\sum_{j=1}^3y_j\kappa_j)\wedge v_0\wedge v_1\wedge v_3\wedge v_4&=0.\nonumber\\
(x_1 \alpha_1+x_2 \alpha_2+\sum_{j=1}^3y_j\kappa_j)\wedge v_0\wedge v_1\wedge v_3\wedge v_5&=0.\nonumber
\end{align}

Let us now consider the rational map $\psi: S_A^{[3]}\to G(3,W)$ defined on general subschemes $\mathfrak{s}\subset S_A$ of length 3 as the $3$-codimensional space in $W^{\vee}=H^0(\mathcal{I}_S(2))$
consisting of those quadrics which contain the plane spanned by $\mathfrak{s}$.  It is clear that for a subscheme corresponding to a general triple of points $\{\beta_1,\beta_2,\beta_3\}$ we have
\begin{equation}\label{psi using phi}
\psi(\{\beta_1,\beta_2,\beta_3\})= [\langle\varphi( \{\beta_1,\beta_2\})\wedge  \varphi( \{\beta_1,\beta_3\})\wedge  \varphi( \{\beta_2,\beta_3\})\rangle].
\end{equation}

\begin{prop} \label{prop specialA} The map $\psi$ is a generically 2:1 rational map onto $D^2_A$.
\end{prop}
\begin{proof} Let $\beta_1$, $\beta_2$, $\beta_3$ be three general points on $S_A$. The proof then amounts to two lemmas:
\begin{lem}\label{two to one} The fiber of $\psi$,
\[ \psi^{-1}(\psi(\{\beta_1,\beta_2,\beta_3\}))=\{\{\beta_1,\beta_2,\beta_3\},\{\gamma_1,\gamma_2,\gamma_3\}\}
\]
 is two triples of points on $S_A$ whose union 
%$\{\beta_1,\beta_2,\beta_3,\gamma_1,\gamma_2,\gamma_3\}\subset S_A$ 
is a set of six distinct points on a twisted cubic contained in $G(3,V)$.
\end{lem}
\begin{proof} Let  $U_{\beta_1},U_{\beta_2},U_{\beta_3}\subset V$ be the subspaces corresponding to $\beta_1$, $\beta_2$, $\beta_3$. Then there exists a unique $3$-dimensional  subspace $U_{\beta_1,\beta_2, \beta_3}$ meeting each $U_{\beta_i}$ in a $2$-dimensional space. It follows that $U_{\beta_1},U_{\beta_2},U_{\beta_3}$ is contained in the intersection $C_{\beta_1,\beta_2,\beta_3}$ of $\mathbb{P}^6$ with the Schubert cycle $\mathcal{S}_{\beta_1,\beta_2,\beta_3}$ in $G(3,V)$
of three-spaces meeting $U_{\beta_1,\beta_2, \beta_3}$ in a $2$-dimensional space. Since $\mathcal{S}_{\beta_1,\beta_2,\beta_3}$ is a cone over $\mathbb{P}^1\times \mathbb{P}^2$ the considered intersection $C_{\beta_1,\beta_2,\beta_3}$ is, in general, a twisted cubic.
Moreover, under the generality assumption $C_{\beta_1,\beta_2,\beta_3}\cap S_A= C_{\beta_1,\beta_2,\beta_3}\cap q^*_A$ consists of six points. Three of them are $\beta_1,\beta_2,\beta_3 $ and the residual three will be denoted by $\gamma_1,\gamma_2,\gamma_3$.  The linear span of $C_{\beta_1,\beta_2,\beta_3}$ is a $\mathbb{P}^3$, we denote it by  $\mathbf{P}$, and its intersection with $G(3,V)$ is 
$\mathbf{P}\cap G(3,V)=C_{\beta_1,\beta_2,\beta_3}$. We denote by $\Pi$ the plane $\langle\beta_1,\beta_2,\beta_3\rangle$. Now, every quadric containing $S_A$ and $\Pi$,
when restricted to $\mathbf{P}$, decomposes into $\Pi$ and another plane $\Pi'$. Since, in general, $\Pi $ does not pass through $\gamma_i$ for $i=1,2,3$, the plane $\Pi' $ must pass through the points $\gamma_i$ for $i=1,2,3$. This means that $\Pi'=\langle\gamma_1,\gamma_2,\gamma_3\rangle$. It is then clear that $\psi(\{\beta_1,\beta_2,\beta_3\})=\psi(\{\gamma_1,\gamma_2,\gamma_3\})$.

Assume on the other hand that $\psi(\{\beta_1,\beta_2,\beta_3\})=\psi(\{\gamma'_1,\gamma'_2,\gamma'_3\})$. 
Then, by the equations \ref{psi using phi} and \ref{eq explicit phi }, we deduce that $U_{\beta_1,\beta_2, \beta_3}=U_{\gamma'_1,\gamma'_2, \gamma'_3}$ 
hence $C_{\beta_1,\beta_2, \beta_3}=C_{\gamma'_1,\gamma'_2, \gamma'_3}$.  It follows that $\langle\gamma'_1,\gamma'_2, \gamma'_3\rangle\subset \mathbf{P}$. 
But the net of quadrics corresponding to  $\psi(\{\beta_1,\beta_2,\beta_3\})=\psi(\{\gamma'_1,\gamma'_2,\gamma'_3\})$ 
define on $\mathbf{P}$ two planes $\langle\gamma_1,\gamma_2,\gamma_3\rangle$ and $\langle\beta_1,\beta_2,\beta_3\rangle$. 
It follows that $\{\gamma'_1,\gamma'_2,\gamma'_3\}=\{\beta_1,\beta_2,\beta_3\}$
or $\{\gamma'_1,\gamma'_2,\gamma'_3\}=\{\gamma_1,\gamma_2,\gamma_3\}$. Which ends the proof.
\end{proof}
\begin{lem}\label{intersection is 2dimensional}   $\dim (T_{\psi(\{\beta_1,\beta_2,\beta_3\})}\cap A)=2$
\end{lem}
\begin{proof} By appropriate choice of basis of $V$ we can assume, without loss of generality, that $\beta_1=v_1\wedge v_2\wedge v_3$, $\beta_2=v_1\wedge v_4 \wedge v_5 $, and $\beta_3=v_2\wedge v_4 \wedge (v_3+v_5)$.
 Observe as above  that $\beta_i \wedge \kappa=0$ for $i=1,2$ and $\kappa\in K$, hence $\beta_i $ is contained
in the space spanned by $A$ and $F_{[v_0]}$. It follows that there exist $\alpha_i \in \wedge^2 V$ such that $\beta_i+v_0\wedge \alpha_i \in A$.  We fix such $\alpha_i$ (determined modulo $K$).
Since $A$ is Lagrangian we have:
%\begin{itemize}
%\item $\alpha_i\wedge \beta_i=0$ for $i=1,2,3$
%\item $\alpha_1\wedge \beta_2=\alpha_2\wedge \beta_1:=c_{12}$
%\item $\alpha_1\wedge \beta_3=\alpha_3\wedge \beta_1:=c_{13}$
%\item $\alpha_2\wedge \beta_3=\alpha_3\wedge \beta_2:=c_{23}$
%\end{itemize}
\[
\alpha_i\wedge \beta_i=0,\quad i=1,2,3, \quad \alpha_1\wedge \beta_2=\alpha_2\wedge \beta_1:=c_{12},
\]
\[
\alpha_1\wedge \beta_3=\alpha_3\wedge \beta_1:=c_{13}\quad {\rm and}\quad 
\alpha_2\wedge \beta_3=\alpha_3\wedge \beta_2:=c_{23}.
\]
As above, a direct computation gives 
\[
\varphi(\{\beta_1,\beta_2\})=c_{12}v_0+v_1,\;\;
\varphi(\{\beta_1,\beta_3\})=c_{13}v_0+v_2,\;{\rm and}\;
\varphi(\{\beta_2,\beta_3\})=-c_{23}v_0+v_4.
\]
It follows that:
\begin{multline*}\textstyle
T_{\psi(\{\beta_1,\beta_2,\beta_3\})}=\{\omega \in \wedge^3 W| \omega\wedge (c_{12}v_0+v_1)\wedge (c_{13}v_0+v_2)= \\
\omega\wedge (c_{12}v_0+v_1)\wedge (-c_{23}v_0+v_4)=
\omega\wedge (c_{13}v_0+v_2)\wedge (-c_{23}v_0+v_4)=0\}.
\end{multline*}
Again we denote by $\kappa_1$, $\kappa_2$, $\kappa_3$ a basis of $K$.
Now, $\beta_i+v_0\wedge \alpha_i \in A$ and $K\wedge v_0\subset A$, so to prove the lemma it is enough to prove that the system of equations
\begin{equation}%\label{system of equations}
\nonumber
\begin{cases}
(\sum_{i=1}^3 x_i (\beta_i+v_0\wedge \alpha_i )+\sum_{j=1}^3 y_j v_0\wedge \kappa_j ) \wedge (c_{12}v_0+v_1) \wedge (c_{13}v_0+v_2)=0\\
(\sum_{i=1}^3 x_i (\beta_i+v_0\wedge \alpha_i )+\sum_{j=1}^3 y_j v_0\wedge \kappa_j ) \wedge (c_{12}v_0+v_1) \wedge (-c_{23}v_0+v_4)=0\\
(\sum_{i=1}^3 x_i (\beta_i+v_0\wedge \alpha_i )+\sum_{j=1}^3 y_j v_0\wedge \kappa_j ) \wedge (c_{13}v_0+v_2) \wedge (-c_{23}v_0+v_4)=0
\end{cases}
\end{equation}
in variables $x=(x_1, x_2, x_3)$, $y=(y_1, y_2, y_3)$ has a $2$-dimensional set of solutions satisfying  $x=(x_1,x_2,x_3)\neq 0$. 
By reductions as above and rearranging we get the system

\begin{equation}\label{newsystem}
 \begin{cases}
  v_0\wedge(-c_{12} x_2\beta_2 \wedge v_2 +c_{13} x_3\beta_3 \wedge v_1+
  (\sum_{i=1}^3 x_i \alpha_i+y_i \kappa_i)\wedge v_1\wedge v_2)=0\\

  v_0\wedge(-c_{12} x_1\beta_1 \wedge v_4 -c_{23}x_3 \beta_3 \wedge v_1+
  (\sum_{i=1}^3 x_i \alpha_i+y_i \kappa_i)\wedge v_1\wedge v_4)=0\\

  v_0\wedge(-c_{13}x_1 \beta_1 \wedge v_4 -c_{23} x_2\beta_2 \wedge v_2+
  (\sum_{i=1}^3 x_i \alpha_i+y_i \kappa_i)\wedge v_2\wedge v_4)=0
\end{cases}
 \end{equation}
 
 To make the system of equations (\ref{newsystem}) into a system of  linear equations we multiply each of the equation by the coordinate vectors and obtain  a  system of 18 linear equations in 6 coordinates. If we now denote the three left hand side expressions dependent on $(x,y)$ in the equations from (\ref{newsystem}) by $u_1(x,y)$, $u_2(x,y)$, $u_3(x,y)\in \wedge^5 W$, a straightforward computation, as above, shows that the following equations are trivial:
 \begin{align*}
 u_1(x,y)\wedge v_0=u_1(x,y)\wedge v_1=u_1(x,y)\wedge v_2=u_1(x,y)\wedge v_3=0,\\
 u_2(x,y)\wedge v_0=u_2(x,y)\wedge v_1=u_2(x,y)\wedge v_4=u_2(x,y)\wedge v_5=0,\\
 u_3(x,y)\wedge v_0=u_3(x,y)\wedge v_2=u_3(x,y)\wedge v_4=u_3(x,y)\wedge (v_3+v_5)=0.\\
  \end{align*}
The following products are equal
  \[
  u_1(x,y)\wedge v_4=-u_2(x,y)\wedge v_2=u_3(x,y)\wedge v_1=(\sum_{i=1}^3 x_i \alpha_i+y_i \kappa_i)\wedge v_0\wedge v_1 \wedge v_2\wedge v_4,
  \]
  while
   \begin{align*}
 u_1(x,y)\wedge v_5&=c_{13}x_3 v_0\wedge \ldots\wedge v_5+(\sum_{i=1}^3 x_i \alpha_i+y_i \kappa_i)\wedge v_0\wedge v_1 \wedge v_2\wedge v_5  \\
 u_2(x,y)\wedge v_3&=c_{23}x_3 v_0\wedge \ldots\wedge v_5-(\sum_{i=1}^3 x_i \alpha_i+y_i \kappa_i)\wedge v_0\wedge v_1 \wedge v_3\wedge v_4,\\
 u_3(x,y)\wedge (v_3-v_5)&=(c_{13}x_1-c_{23}x_2) v_0\wedge \ldots\wedge v_5\\
 &- (\sum_{i=1}^3 x_i \alpha_i+y_i \kappa_i)\wedge v_0\wedge v_2 \wedge v_4\wedge (v_3-v_5).\\
 \end{align*}
So the 18 linear equations are reduced the following four independent ones:
 \begin{align*}
 u_1(x,y)\wedge v_4=0, u_1(x,y)\wedge v_5=0,\\
 u_2(x,y)\wedge v_3=0, u_3(x,y)\wedge (v_3-v_5)=0.
 \end{align*}

  It follows that the system of linear equations admits a $2$-dimensional system of solutions.  To prove that nonzero solutions satisfy $x\neq 0$ it is enough to observe that a solution with $x=0$ is a $3$-vector $v_0\wedge\kappa$ with $\kappa\in K$ such that 
 $$\kappa\wedge v_1\wedge v_2=\kappa\wedge v_1\wedge v_4= \kappa\wedge v_2\wedge v_4=0.$$
But any such $\kappa$ lies in the space $\langle v_1\wedge v_2, v_1\wedge v_4,  v_2\wedge v_4\rangle =\wedge^2\langle v_1,v_2,v_4\rangle$.  By assumption, $\PP(v_0\wedge K)\subset \PP(A)$ does not intersect $G(3,W)$, so this is impossible.
   Therefore the only solution of the system (\ref{newsystem}) satisfying $x=0$ is $(x,y)=(0,0)$.
\end{proof}
Proposition \ref{prop specialA} follows immediately from Lemmas \ref{two to one} and \ref{intersection is 2dimensional}.\end{proof}

\begin{rem} There is an alternative approach to the Proposition \ref{prop specialA}.
We consider the intersection $F$ of the quadrics containing $S_A$ together with a generic plane $B=\langle\beta_1,\beta_2,\beta_3  \rangle$.
This is a complete intersection of degree $8$ with six ordinary double points that span a $3$-space.  Three of them are the points of intersection of $B\cap S_A$ and the residual three points span another plane $B'$ contained in our Fano threefold  $F$.
Since $S_A$ does not contain any plane curve, if a plane passes through three points of $S_A$ these points are isolated in the intersection.
 If the plane is contained in $F$, the three points must therefore be three of the six ordinary double points.  Since the 3-space spanned by $B$ and $B'$ cuts $F$ along the sum of $B\cup B'$ it follows that the degree of $\psi$ is two at the point corresponding to $F$. % \krtodo{This does not make sense without precise additional claims}
 On the other hand the generic complete intersection of three quadrics containing $S_A$ have also six ordinary double points.
The six ordinary double points spans a $\PP^3$ and a complete intersection $F$ of degree $8$ that contain $S_A$ corresponds to a point of the EPW cube when the intersection
of this $\PP^3$ with $F$ is a reducible quadric. 
\end{rem}
Next, we compute the codimension of the indeterminacy locus and the ramification locus of $\psi$.
\begin{prop} \label{exceptional} The rational map $\psi$ is well defined outside a set of codimension 2.
Moreover, the ramification locus of $\psi$ is of codimension $\geq 2$.
\end{prop}
Before we pass to the proof of the Proposition we introduce some more notation. Recall first that, by the assumption on generality of $A$, we know that $S_A$ does not contain any line, conic or twisted  cubic. 
Let $F_A$ be the Fano threefold obtained as the intersection $G(3,V)\cap \langle S_A\rangle$. By the generality of $A$, it follows that $F_A$ is smooth.  Let $[U]\in G(3,V)$. Consider the Schubert cycle $\mathcal{S}_{U}=\{U'\in G(3,V)| \dim(U\cap U')\geq 2\}$. It is clear that in the Pl\"ucker embedding of $G(3,V)\subset \mathbb{P}(\wedge^3 V)$ the variety  $\mathcal{S}_{U}$ is the tangent cone of $G(3,V)$ in $[U]$. It spans the projective  tangent space and is a cone over $\mathbb{P}^1\times \mathbb{P}^2$ with vertex $[U]$. 
We are interested in intersections $\mathcal{S}_{U}\cap F_A$. Note that $F_A$ is of degree 5 and has Picard group of rank 1 generated by the hyperplane class. Hence $F_A$ does not contain any surface of degree $\leq 4$. It follows that $\mathbb{P}^6_A\cap \mathcal{S}_{U}=F_A \cap \mathcal{S}_{U}$ is a cubic curve,   a possibly reducible or nonreduced degeneration of a twisted cubic curve. We denote the corresponding subscheme of the Hilbert scheme of twisted cubics in $F_A$ by $\mathcal{H}_A$.

Let $\mathcal{B}_1$ be the subset of $S_A^{[3]}$ consisting of those subschemes that are contained in a conic in $F_A\subset G(3,V)$. Since $F_A $ is a linear section of $G(3,V)$ and contains no planes, the Hilbert scheme of conics in $F_A$ admits a birational map to $\mathbb{P}(V)$ associating to a conic $\mathfrak{c}$ the intersection of three-spaces parametrized by points on $\mathfrak{c}$. It is hence of dimension 4 and we get that $\mathcal{B}_1$ is of dimension 4.
Let  $\mathcal{B}_2$ be the subset of $S^{[3]}_A$ consisting of those subschemes that meet some line contained in $G(3,V)$ in a scheme of length two. Then $\mathcal{B}_2$ is also of dimension 4, since the Hilbert scheme of lines in $F_A$ isomorphic to $\mathbb{P}^2$ (cf. \cite[Proposition 5.2]{EPWMichigan}, \cite{Iskovskih}). 
\begin{lem}\label{unique twisted cubic} Let $\mathfrak{s}$ be a subscheme of length 3 in $S_A$ corresponding to a point from $S_A^{[3]}\setminus(\mathcal{B}_1\cup\mathcal{B}_2 )$. Then there is a unique, possibly degenerate, twisted cubic from $\mathcal{H}_A$ that contains $\mathfrak{s}$. Furthermore, the induced map $S_A^{[3]}\setminus(\mathcal{B}_1\cup\mathcal{B}_2 )\to \mathcal{H}_A$ is dominant.
\end{lem}
\begin{proof}
Since $S_A\subset G(3,V)\cong G(2,V^{\vee})$, we may characterize the elements of $\sigma\in S_A^{[3]}$ via the incidence of curves $C_\sigma$ of degree 3 in $\PP(V^{\vee})$ supported on lines.  For a general $\sigma$,  the curve $C_\sigma$ is the union of three lines and has a unique transversal line, a line that meet all three lines.  If $\sigma \in S_A^{[3]}\setminus(\mathcal{B}_1\cup\mathcal{B}_2 )$ the curve $C_\sigma$ spans  $\mathbb{P}(V^{\vee})$ and contains no conic. It follows that $C_\sigma$ admits a unique transversal line hence $\mathfrak{s}_{\sigma}$ is contained in $\mathcal{S}_{U}$ for a unique $U$. We conclude by the definition of $\mathcal{H}_A$. For dominancy of the map we observe that if $\mathfrak{c}\in \mathcal{H}_A$ then $\mathfrak{c}\cap q^*_A\subset S_A$ and clearly contains a subscheme in $S_A^{[3]}\setminus(\mathcal{B}_1\cup\mathcal{B}_2)$ .
\end{proof}
We can now pass to the proof of Proposition \ref{exceptional}
\begin{proof}[Proof of Proposition \ref{exceptional}]  Any subscheme $\mathfrak{s}$ of length 3 in $S_A$ spans a plane $\Pi_{\mathfrak{s}}$. The map $\psi$ associates to $\mathfrak{s}$ the space $V^q_{\mathfrak{s}}$ of quadrics containing $S_A\cup \Pi_{\mathfrak{s}}$. For general $\mathfrak{s}$ the latter is a space of dimension 3. Now, $\psi$ is well defined exactly on those $\mathfrak{s}$ for which $\dim V^q_{\mathfrak{s}}=3$. But $V^q_{\mathfrak{s}}$
is the kernel of the restriction map $H^0(S_A,\mathcal{I}_{S_A}(2))\to H^0(\Pi_{\mathfrak{s}}, \mathcal{I}_{S_A\cap \Pi_{\mathfrak{s}}}(2)).$ The latter kernel is 3-dimensional unless  $\dim H^0(\Pi_{\mathfrak{s}}, \mathcal{I}_{S_A\cap \Pi_{\mathfrak{s}}}(2))\leq 2$. Hence $\psi$ is not defined only if
$S_A\cap \Pi_{\mathfrak{s}}$ has length at least 4. Then the intersection $\Pi_{\mathfrak{s}}\cap G(3,V)$ contains a scheme of length 4. As $S_A$ contains no conics,  $\Pi_{\mathfrak{s}}$ cannot be contained in $G(3,V)$. We infer by \cite[proof of Lemma 2.2]{MukaiGrassmannian} that $\Pi_{\mathfrak{s}}\cap G(3,V)$ contains a line or a unique conic.  If $\Pi_{\mathfrak{s}}\cap G(3,V)$ contains a line, then it is either a reducible conic or the union of this line with a point. In the latter case, since $S_A$ contains no lines,  the intersection $\Pi_{\mathfrak{s}}\cap S_A$ does not contain any subscheme of length 4. It follows that there is a map with finite fibers from the indeterminacy locus of $\psi$ to the  Hilbert scheme of conics in  $G(3,V)\cap \mathbb{P}^6$ which is of dimension 4. 
We conclude that the indeterminacy locus is of dimension at most 4. In fact it is equal to 4 since a general $V_4\subset V$ defines a conic in $G(3,V)\cap \mathbb{P}^6$ which meets $S_A$ in four points.

Finally to bound the dimension of the ramification locus, we again let $\mathfrak{s}$ be a subscheme of length 3 in $S_A$ corresponding to a point from $S_A^{[3]}\setminus(\mathcal{B}_1\cup\mathcal{B}_2 )$. Then by Lemma \ref{unique twisted cubic} there is a possibly degenerate twisted cubic from $\mathcal{H}_A$ spanning a $\mathbb{P}^3$ and containing $\mathfrak{s}$. Now, from the proof of Proposition \ref{prop specialA} we know that a point from $S_A^{[3]}\setminus(\mathcal{B}_1\cup\mathcal{B}_2 )$ can be in the ramification locus of $\psi$ only if the quadric $Q_A$ is totally tangent to the twisted cubic. The latter is a codimension 3 condition on twisted cubics in $G(3,V)\cap \mathbb{P}^6$, hence  by Lemma \ref{unique twisted cubic} a codimension 3 condition for the ramification locus. To be more precise we have an incidence:
$$\mathcal{X}=\{(C,Q)\in \mathcal{H}_A\times H^0(O_{\PP^6}(2))|\;\; Q|_C\;\; \text{is totally non reduced}\}.$$
 We compute its dimension from the projection onto $\mathcal{H}_A$. Indeed, fixing $C$ we get  a codimension 3 space of quadrics totally tangent to it. The dimension of the general fiber of the second projection follows giving codimension 3 in $\mathcal{H}_A$.
\end{proof}

\section{the proof of Theorem \ref{main}}\label{proof}
Let us choose a generic Lagrangian space $A_0$ satisfying $[A_0]\in \Delta \setminus (\Gamma\cup \Sigma)\subset LG_\eta(10,\wedge^3W)$.
Note that from Lemma \ref{distinct delta gamma}, we can choose $A_0$ such that $K$ is generic in $F_{[v_0]}$.
From Proposition \ref{prop specialA} there is a rational $2:1$ map  $\psi: S_{A_0}^{[3]}\to D^{A_0}_2$.
On the other hand from Proposition \ref{existence and smoothness of double cover} there exists a double cover $Y_{A_0}\to D^{A_0}_2$ such that $Y_{A_0}$ is a smooth sixfold with trivial canonical bundle.
Our aim is to construct a birational map \[ S_{A_0}^{[3]}\dashrightarrow X_{A_0}. \]

We consider the subset $\mathcal{B}$ in $S_{A_0}^{[3]}$, the union of the indeterminacy locus and the ramification locus of the rational $2:1$ map $\psi \colon S_{A_0}^{[3]}\to D_2^{A_0}$. Clearly the restriction of the map $\psi$ to $S_{A_0}^{[3]}\setminus \mathcal{B}$ is an \'etale covering of degree $2$ onto a smooth open subset $\mathcal{D}\subset D_2^{A_0}$. In particular $\mathcal{D}\cap D_3^{A_0}=\emptyset.$
Note that  $S_{A_0}^{[3]}$ is simply connected and by Proposition \ref{exceptional} the subset $\mathcal{B}$ is of codimension $2$. This implies that $S_{A_0}^{[3]}\setminus \mathcal{B}$ is also simply connected.
It follows that $\pi_1(\mathcal{D})=\mathbb{Z}_2$ and $\psi |_{S_{A_0}^{[3]}\setminus \mathcal{B}}$ is a universal covering.

Since  $\mathcal{D}$ is disjoint from $D_3^{A_0}$, the restriction of the double cover $f_{A_0}\colon Y_{A_0}\to D_2^{A_0}$ to $f_{A_0}^{-1}(\mathcal{D})$ is also an \'etale covering.

By Proposition \ref{existence and smoothness of double cover} the variety $Y_{A_0}$
is smooth and irreducible. It follows that the \'etale covering $f_{A_0}|_{f_{A_0}^{-1}(\mathcal{D})}$ is not trivial.
 We infer that $f_{A_0}|_{f_{A_0}^{-1}(\mathcal{D})}$ is also the universal covering, and deduce that $Y_{A_0}$ is birational to $S_{A_0}^{[3]}$. 

Note that the fact that $f_{A_0}|_{f_{A_0}^{-1}(\mathcal{D})}$ is the universal covering implies $f_{A_0}^{-1}(\mathcal{D})$  is simply connected. It follows that $Y_{A_0}$ is also simply connected because $f_{A_0}^{-1}(\mathcal{D})$ is obtained from the smooth variety $Y_{A_0}$ by removing a subset of real codimension 2. Moreover, since both $Y_{A_0}$ and $S^{[3]}$ have trivial canonical bundle, by \cite[Theorem~1.1]{Ito} they have equal Hodge numbers. Thus $$h^{2}(\oo_{Y_{A_0}})=h^2(\oo_{S_{A_0}^{[3]}})=1.$$
From the Beauville classification theorem \cite[Theorem~2]{Beauville} we infer that $Y_{A_0}$ is IHS.

Recall the notation $$LG_\eta^1(10,\wedge^3 W):=\{[A]\in LG_{\eta}(10,\wedge^3 W)| \mathbb{P}(A)\cap G(3,W)=\emptyset, \forall [U]\in G(3,W): \dim(A\cap T_U)\leq 3 \}.$$ Consider now the varieties: 
$$\mathcal{D}_k=\{ ([A],[U])\in LG_\eta^1(10,\wedge^3 W)\times G(3,W)| [U]\in D_k^{A}  \},$$
for $k=2,3$. By globalizing the construction in Proposition \ref{existence and smoothness of double cover} to the affine variety $LG_\eta^1(10,\wedge^3 W)$ we construct a variety $\mathcal{Y}$ which is a double cover of $\mathcal{D}_2 $ branched in $\mathcal{D}_3$. We get a smooth family $$\mathcal{Y}\to LG_{\eta}^1(10,\wedge^3 W) $$
 with fibers  $\mathcal{Y}_{[A]}=Y_{A}$ polarized by the divisor defining the double cover. In particular a special fiber $\mathcal{Y}_{[A_0]}=Y_{A_0}$ is an IHS manifold. 
Since a smooth deformation of an IHS manifold is still IHS we obtain that $Y_A$ is IHS for every $A\in LG_\eta^1(10,\wedge^3 W)$. So $\mathcal{Y}\to LG_\eta^1(10,\wedge^3 W)$ is a family of IHS manifolds.

In order to show that the IHS sixfolds in the family  $\mathcal{Y}$ are of $K3^{[3]}$-type we use the fact proved above that $S_{A_0}^{[3]}$ and $Y_{A_0}$ are birational. Indeed, two birational IHS manifolds are deformation equivalent from \cite[Theorem~4.6]{Huybrechts}. The Beauville-Bogomolov degree $q=4$ of our polarization follows from our computation of degree in Section \ref{invariants}.

We end the proof of Theorem \ref{main} by performing a study of the moduli map defined by the family $\mathcal{Y}$.

\begin{prop} Let $\mathcal{M}$ be the coarse moduli space of polarized  IHS %hyperk\"{a}her 
sixfolds of $K3^{[3]}$-type and Beauville-Bogomolov degree $4$. Let  
$$\mathfrak{M}_{\mathcal{Y}} :LG_{\eta}^1(10,\wedge^3 W)\to \mathcal{M}, \quad [A] \mapsto [Y_{A}] $$
 be the map given by $\mathcal{Y}$. The image of $\mathfrak{M}_{\mathcal{Y}}$ is a dense open subset of a component of dimension 20 in $\mathcal{M}$. 
\end{prop} 
For the proof we will need the following lemma.
\begin{lem}  \label{intersection of grassmannians} Let $A\in LG_{\eta}^1(10,\wedge^3 W)$. If $\mathfrak{g}\in PGL(\wedge^3 W)$ is such that $D^A_2\subset G(3,W) \cap \mathfrak{g}(G(3,W))$,  then $G(3,W)=\mathfrak{g}(G(3,W))$.  
\end{lem}
\begin{proof}Let us denote by $G_1$, $G_2$ the varieties $G(3,W)$ and $\mathfrak{g} (G(3,W))$ respectively.

Let $X\subset G_1\cap G_2$ be an irreducible component of the intersection that contains $D^2_A$.
Then $X$ has codimension at most $3$ in both $G_1$ and $G_2$ and spans  $\PP^{19}$.  Furthermore it is contained in a complete intersection of quadric hypersurfaces on each $G_i$.
If $X$ has codimension $3$, then $X=D^A_2$ and lies in a complete intersection of three quadrics.  But the complete intersection has degree $8\cdot 42=336$, while $D^A_2$ has degree $480$, so this is impossible.  

For lower codimension of $X$ we first note that $D^A_2\subset D^A_1$.  Since  $[D^A_1]=[c_1( \mathcal{T}^{\vee})\cap G(3,W)]$ and $c_1( \mathcal{T}^{\vee})=4h$, the divisor $D^A_1$ is a quartic hypersurface section of $G_1$ and $G_2$. So %when $X$ has codimension two or one in the $G_i$, 
we may assume that $D^A_2$ is contained in a quartic hypersurface section of $X$.

Consider the following subvariety in $G_1$:  Let $V_5\subset W$ be a general $5$-dimensional subspace, and let $V_1$ be a general $1$-dimensional subspace of $V_5$.
Let $F(1,5)=\{[U]\in G_1|V_1\subset U\subset V_5\}\subset G_1$ and denote by $P(1,5)$ the span of $F(1,5)$. %=\{[U]\in G_1|V_1\subset U\subset V_5\}\subset G_1$.
Then $F(1,5)$ is a $4$-dimensional smooth quadric and the span, $P(1,5)$, is a $\PP^5$.  

If $X$ has codimension $2$, then $X_{(1,5)}:=X\cap F(1,5)$ is an irreducible surface.  Furthermore,  $X_{(1,5)}$ is contained in at least $2$ quadric sections of $F(1,5)$.  So $X_{(1,5)}$ has degree at most $8$.
On the other hand $$D_{(1,5)}:=D^A_2\cap F(1,5)\subset X_{(1,5)} $$ is a curve of degree $56$, contained in a quartic hypersurface section of $X_{(1,5)}$, which has degree at most $32$.   
Since this is absurd, we may assume that $X$ has codimension one, i.e. is a divisor in the $G_i$.

Since $D^A_2$ spans $\PP^9$, the divisor $X$ must be a quadric hypersurface section of each $G_i$.   Then $P(1,5)\cap X$ is complete intersection of two quadrics, and through every point of $P(1,5)$ there are inifinitely many secants lines to $X$. 
The union of the spaces $P(1,5)$ as $V_5$ and $V_1$ varies is a variety $\Omega_1\subset \PP^{19}$,  characterized in  \cite[Lemma 3.3]{Donagi} as the locus of points in $\PP^{19}$ that lies on more than one secant line to $G_1$.  Furthermore $G_1$ is the singular locus of $\Omega_1$.  Similarly, $\Omega_2$ is defined with respect to $G_2$.  By the above argument each $P(1,5)$ in $\Omega_1$ is also contained in $\Omega_2$.  Thus  $\Omega_1\subset \Omega_2$.  But then they coincide, and since $G_i={\rm Sing}(\Omega_i)$, the two grassmannians $G_1$ and $G_2$ coincide.

 \end{proof}
 
\begin{proof}

We claim that $\mathfrak{M}_{\mathcal{Y}}([A_1])=\mathfrak{M}_{\mathcal{Y}}([A_2])$ if and only if there exists a linear automorphism $\mathfrak{g}\in Aut(G(3,W))\simeq \mathbb{Z}/2\times PGL(W)$ such that $\mathfrak{g}(A_1)= A_2$. 
Indeed, assume that $\mathfrak{M}_{\mathcal{Y}}([A_1])=\mathfrak{M}_{\mathcal{Y}}([A_2])$. Then $Y_{A_1}$ and $Y_{A_2}$, polarized by ample classes defining double covers to
$D^2_{A_1}$ and $D^2_{A_2}$  respectively, are isomorphic. It follows that there is a linear automorphism $\mathfrak{g}\in PGL(\wedge^3 W)$ such that $\mathfrak{g}(D^2_{A_1}) =D^2_{A_2}$. It follows that $D^2_{A_2}\subset  G(3,W)\cap \mathfrak{g}(G(3,W))$. By Lemma \ref{intersection of grassmannians}, we deduce that $G(3,W)=\mathfrak{g}(G(3,W))$. It follows that $\mathfrak{g}\in Aut(G(3,W))$.

By \cite{Ogrady-IHS} the locus $LG^1_\eta(10,\wedge^3W)$ is contained in the stable locus of the natural linearized $PGL(W)$ action on $LG_\eta(10,\wedge^3W)$.
From our claim we hence infer that 
$$\dim (\mathfrak{M}_{\mathcal{Y}}(LG^1_\eta(10,\wedge^3 W)))\geq  \dim LG_\eta^1(10,\wedge^3 W) - \dim (PGL(W))=55 - 35= 20.$$ 
But 20 is the dimension of $\mathcal{M}$, so our map is surjective onto an (also by stability) open subset of a component of $\mathcal{M}$ of dimension 20. 
\end{proof}
We conclude by determining the component of the moduli space that is filled by our family.  

Recall that for $v\in H^2((K3)^{[3]},\Z)$ the divisibility of $v$ is defined as the generator of the subgroup 
$(v,H^2((K3)^{[3]},\Z))\subset \Z$ where $(.,.)$ is the scalar product induced by the Beauville-Bogomolov form.
Note that for Beauville-Bogomolov degree $4$ there are two possible divisibilities for $H$ either $l=1$ or $2$ (see \cite[Proposition~3.6]{GritsenkoHulekSankaran}). It follows from \cite[Proposition~2.1(3) and Corollary~2.4]{Apostolov} that there are  exactly two components, distinguished by the divisibility, of the coarse moduli space of polarized  IHS %hyperk\"{a}her 
sixfolds of $K3^{[3]}$-type and Beauville-Bogomolov degree $4$.  Which one is determined by the following proposition, whose proof was pointed out to us by Kieran O'Grady.
\begin{prop} \label{divisibility}The image of $\mathfrak{M}_{\mathcal{Y}}$ is open and dense in the connected component of the coarse
moduli space of IHS 
%hyperk\"{a}her 
sixfolds of $K3^{[3]}$-type, Beauville-Bogomolov degree $4$ and divisibility 2; 
\end{prop} 
\begin{proof} By the above, it remains to compute the divisibility of our polarization. For this, fix $A$ general and denote the polarization by $P$. Observe that the involution of the double cover $Y_A \to D^2_A$
defined by the polarization is anti-symplectic. Indeed as an involution on an IHS manifold it is either symplectic or anti-symplectic, but the fixed point locus of a symplectic
involution is a symplectic manifold (see \cite[Proposition 3]{Camere}) whereas the fixed locus of our involution is of dimension 3. This means that the involution must be anti-symplectic . 
Moreover, since we proved that our family is of maximal dimension, we may assume that $Y_A$ has Picard group spanned by the polarization $P$. It follows that the action of the involution on $H^2(Y_A)$ has an invariant subspace spanned by the class $[P]$. Furthermore, the involution 
respects the Beauville-Bogomolov bilinear form $(.,.)$. Thus, since $([P],[P])=4$, the involution on $H^2(Y_A)$ is of the form 
$$v\mapsto -v+\frac{1}{2}(v,[P])[P].$$
 Since the involution must map integral cohomology to integral cohomology, it follows that $(v,[P])$ is even for all integral classes $v$. This implies that the divisibility of $[P]$ is not equal to 1. We infer that it is equal to  2.   
\end{proof}

\bibliography{biblio} \bibliographystyle{alpha}

\begin{thebibliography}{LLSvS15}

\bibitem[Apo11]{Apostolov}
Apostol Apostolov.
\newblock Moduli spaces of polarized irreducible symplectic manifolds are not
  necessarily connected.
\newblock {\em arXiv:1109.0175}, 2011.

\bibitem[AR04]{AR}
Alberto Alzati and Francesco Russo.
\newblock Some extremal contractions between smooth varieties arising from
  projective geometry.
\newblock {\em Proc. London Math. Soc. (3)}, 89(1):25--53, 2004.

\bibitem[BD85]{BeauvilleDonagi}
Arnaud Beauville and Ron Donagi.
\newblock La vari\'et\'e des droites d'une hypersurface cubique de dimension
  {$4$}.
\newblock {\em C. R. Acad. Sci. Paris S\'er. I Math.}, 301(14):703--706, 1985.

\bibitem[Bea83]{Beauville}
Arnaud Beauville.
\newblock Vari\'et\'es {K}\"ahleriennes dont la premi\`ere classe de {C}hern
  est nulle.
\newblock {\em J. Differential Geom.}, 18(4):755--782 (1984), 1983.

\bibitem[Cam12]{Camere}
Chiara Camere.
\newblock Symplectic involutions of holomorphic symplectic four-folds.
\newblock {\em Bull. Lond. Math. Soc.}, 44(4):687--702, 2012.

\bibitem[Don77]{Donagi}
Ron~Y. Donagi.
\newblock On the geometry of {G}rassmannians.
\newblock {\em Duke Math. J.}, 44(4):795--837, 1977.

\bibitem[DV10]{DV}
Olivier Debarre and Claire Voisin.
\newblock Hyper-k\"ahler fourfolds and {G}rassmann geometry.
\newblock {\em J. Reine Angew. Math.}, 649:63--87, 2010.

\bibitem[GHS07]{GHS}
V.~A. Gritsenko, K.~Hulek, and G.~K. Sankaran.
\newblock The {K}odaira dimension of the moduli of {$K3$} surfaces.
\newblock {\em Invent. Math.}, 169(3):519--567, 2007.

\bibitem[GHS10]{GritsenkoHulekSankaran}
V.~Gritsenko, K.~Hulek, and G.~K. Sankaran.
\newblock Moduli spaces of irreducible symplectic manifolds.
\newblock {\em Compos. Math.}, 146(2):404--434, 2010.

\bibitem[Har70]{HartshorneLN}
Robin Hartshorne.
\newblock {\em Ample subvarieties of algebraic varieties}.
\newblock Lecture Notes in Mathematics, Vol. 156. Springer-Verlag, Berlin-New
  York, 1970.
\newblock Notes written in collaboration with C. Musili.

\bibitem[Har77]{Hartshorne}
Robin Hartshorne.
\newblock {\em Algebraic geometry}.
\newblock Springer-Verlag, New York-Heidelberg, 1977.
\newblock Graduate Texts in Mathematics, No. 52.

\bibitem[Huy99]{Huybrechts}
Daniel Huybrechts.
\newblock Compact hyper-{K}\"ahler manifolds: basic results.
\newblock {\em Invent. Math.}, 135(1):63--113, 1999.

\bibitem[IR01]{IlievRanestad}
Atanas Iliev and Kristian Ranestad.
\newblock {$K3$} surfaces of genus 8 and varieties of sums of powers of cubic
  fourfolds.
\newblock {\em Trans. Amer. Math. Soc.}, 353(4):1455--1468, 2001.

\bibitem[Isk77]{Iskovskih}
V.~A. Iskovskih.
\newblock Fano threefolds. {I}.
\newblock {\em Izv. Akad. Nauk SSSR Ser. Mat.}, 41(3):516--562, 717, 1977.

\bibitem[Ito03]{Ito}
Tetsushi Ito.
\newblock Birational smooth minimal models have equal {H}odge numbers in all
  dimensions.
\newblock In {\em Calabi-{Y}au varieties and mirror symmetry ({T}oronto, {ON},
  2001)}, volume~38 of {\em Fields Inst. Commun.}, pages 183--194. Amer. Math.
  Soc., Providence, RI, 2003.

\bibitem[Kap14]{GKapustkab23}
Grzegorz Kapustka.
\newblock On ihs fourfolds with $b_2=23$.
\newblock {\em arXiv:1403.1074, The Michigan Mathematical Journal 65 (1), 3-33}, 2014.

\bibitem[LLSvS15]{LehnLehnSorgervanStraten}
Christian Lehn, Manfred Lehn, Christoph Sorger, and Duco van Straten.
\newblock Twisted cubics on cubic fourfolds.
\newblock {\em J. reine angew. Math.}, 2015.

\bibitem[Mor82]{Mori}
Shigefumi Mori.
\newblock Threefolds whose canonical bundles are not numerically effective.
\newblock {\em Ann. of Math.}, 116:133--176, 1982.

\bibitem[Muk92]{M1}
Shigeru Mukai.
\newblock Polarized {$K3$} surfaces of genus {$18$} and {$20$}.
\newblock In {\em Complex projective geometry ({T}rieste, 1989/{B}ergen,
  1989)}, volume 179 of {\em London Math. Soc. Lecture Note Ser.}, pages
  264--276. Cambridge Univ. Press, Cambridge, 1992.

\bibitem[Muk93]{MukaiGrassmannian}
Shigeru Mukai.
\newblock Curves and {G}rassmannians.
\newblock In {\em Algebraic geometry and related topics ({I}nchon, 1992)},
  Conf. Proc. Lecture Notes Algebraic Geom., I, pages 19--40. Int. Press,
  Cambridge, MA, 1993.

\bibitem[Muk06]{M2}
Shigeru Mukai.
\newblock Polarized {$K3$} surfaces of genus thirteen.
\newblock In {\em Moduli spaces and arithmetic geometry}, volume~45 of {\em
  Adv. Stud. Pure Math.}, pages 315--326. Math. Soc. Japan, Tokyo, 2006.

\bibitem[Muk10]{Mukai}
Shigeru Mukai.
\newblock Curves and symmetric spaces, {II}.
\newblock {\em Ann. of Math. (2)}, 172(3):1539--1558, 2010.

\bibitem[Muk12]{M3}
Shigeru Mukai.
\newblock K3 surfaces of genus 16.
\newblock {\em RIMS 1743}, 2012.

\bibitem[O'G06]{Ogrady-IHS}
Kieran~G. O'Grady.
\newblock {Irreducible symplectic 4-folds and {Eisenbud}-{Popescu}-{Walter}
  sextics}.
\newblock {\em Duke Math. J.}, 134(1):99--137, 2006.

\bibitem[O'G10]{OGradyTaxonomy}
Kieran~G. O'Grady.
\newblock Epw-sextics: taxonomy.
\newblock {\em arXiv:1007.3882 [math.AG]}, 2010.

\bibitem[O'G13]{EPWMichigan}
Kieran~G. O'Grady.
\newblock Double covers of {E}{P}{W}-sextics.
\newblock {\em Michigan Math. J.}, 62:143--184, 2013.

\bibitem[PR97]{PragaczRatajski}
P.~Pragacz and J.~Ratajski.
\newblock Formulas for {L}agrangian and orthogonal degeneracy loci;
  {$Q$}-polynomial approach.
\newblock {\em Compositio Math.}, 107(1):11--87, 1997.

\bibitem[Pra91]{PragaczLN1478}
Piotr Pragacz.
\newblock Algebro-geometric applications of {S}chur {$S$}- and
  {$Q$}-polynomials.
\newblock In {\em Topics in invariant theory ({P}aris, 1989/1990)}, volume 1478
  of {\em Lecture Notes in Math.}, pages 130--191. Springer, Berlin, 1991.

\end{thebibliography}
\end{document}